\documentclass[a4paper,10pt]{amsart}

\usepackage[colorlinks=true, linkcolor={black},citecolor={black}]{hyperref}% without colors

\usepackage{verbatim,latexsym,exscale,amssymb,amsthm,amsmath,amsfonts,amscd,mathrsfs,amsrefs,color}
\usepackage{graphicx} % for the command \rotatebox
\usepackage{stmaryrd}
\usepackage[all]{xy}
\usepackage{tikz}
\xyoption{ps}
\SelectTips{cm}{}

\newdir{ >}{{}*!/-6pt/@{>}}

\definecolor{ForestGreen}{rgb}{0.13,.57,.13}

\newcommand{\cA}{\mathcal{A}}
\newcommand{\cB}{\mathcal{B}}
\newcommand{\cC}{\mathcal{C}}
\newcommand{\cD}{\mathcal{D}}
\newcommand{\cE}{\mathcal{E}}
\newcommand{\cF}{\mathcal{F}}
\newcommand{\cG}{\mathcal{G}}
\newcommand{\cI}{\mathcal{I}}
\newcommand{\cL}{\mathcal{L}}

\newcommand{\cR}{\mathcal{R}}
\newcommand{\cS}{\mathcal{S}}
\newcommand{\cT}{\mathcal{T}}

\newcommand{\Loc}{\mathsf{Loc}}
\newcommand{\KK}{\mathsf{KK}}
\newcommand{\FK}{{F\!K}} % filtrated K-theory

\newcommand{\incoming}{\mathrm{N}_{-}} % incoming neighbourhood
\newcommand{\outcoming}{\mathrm{N}_{+}} % outcoming neighbourhood

\newcommand{\base}{\mathbb{K}} %the base commutative ring
\newcommand{\Fun}{\mathsf{Fun}}

\newcommand{\Mod}{\mathop{\mathsf{Mod}}}
\newcommand{\modu}{\mathop{\mathsf{mod}}}
\newcommand{\add}{\mathop{\mathsf{add}}}
\newcommand{\GrMod}{\mathop{\mathsf{GrMod}}}
\newcommand{\Proj}{\mathop{\mathsf{Proj}}}
\newcommand{\proj}{\mathop{\mathsf{proj}}}
\newcommand{\GProj}{\mathop{\mathsf{GProj}}}
\newcommand{\Gproj}{\mathop{\mathsf{Gproj}}}
\newcommand{\stGProj}{\mathop{\underline{\mathsf{GProj}}}}
\newcommand{\stGproj}{\mathop{\underline{\mathsf{Gproj}}}}
\newcommand{\GInj}{\mathop{\mathsf{GInj}}}

\newcommand{\FD}{\mathop{\mathsf{FD}}}
\newcommand{\fd}{\mathop{\mathsf{fd}}}

\newcommand{\Ch}{\mathsf{Ch}}
\newcommand{\D}{\mathsf{D}}
\newcommand{\K}{\mathsf{K}}

\newcommand{\Spec}{\mathop{\mathrm{Spec}}}

\newcommand{\id}{\mathrm{id}}
\newcommand{\N}{\mathbb{N}}
\newcommand{\op}{\mathrm{op}}

\newcommand{\ev}{\mathrm{ev}}

\newcommand{\cone}{\mathop{\mathrm{cone}}}

\newcommand{\colim}{\mathop{\mathrm{colim}}}

\newcommand{\obj}{\mathop{\mathrm{ob}}}

\newcommand{\coker}{\mathop{\mathrm{coker}}}
\newcommand{\Img}{\mathop{\mathrm{im}}} % image of a morphism
\newcommand{\End}{\mathrm{End}}

\newcommand{\inv}{^{-1}}

\newcommand{\Hom}{\mathrm{Hom}}
\newcommand{\Ext}{\mathrm{Ext}}
\newcommand{\Tor}{\mathrm{Tor}}
\newcommand{\Aut}{\mathrm{Aut}}
\newcommand{\Db}{\mathrm{D^b}}

\newcommand{\Z}{\mathbb{Z}}
\newcommand{\Q}{\mathbb{Q}}

\newcommand{\C}{\mathbb{C}}

\renewcommand{\hom}{\mathrm{Hom}}

\newcommand{\Ho}{\mathop{\mathsf{Ho}}}

\DeclareMathOperator{\pdim}{pdim}
\DeclareMathOperator{\idim}{idim}
\DeclareMathOperator{\gldim}{gldim}
\DeclareMathOperator{\Gpdim}{Gpdim}

\renewcommand{\iff}{if and only if }

\numberwithin{equation}{section}

\theoremstyle{definition}

\newtheorem{defi}[equation]{Definition}%[section]
\newtheorem*{conv*}{Conventions}

\newtheorem{hyp}[equation]{Hypothesis}
\newtheorem*{hyp*}{Hypothesis}
\newtheorem{example}[equation]{Example}

\newtheorem{construction}[equation]{Construction}
\newtheorem*{conventions*}{Conventions}

\theoremstyle{plain}

\newtheorem{thm}[equation]{Theorem}
\newtheorem{lemma}[equation]{Lemma}
\newtheorem{thm-defi}[equation]{Theorem-Definition}
\newtheorem{prop}[equation]{Proposition}
\newtheorem{cor}[equation]{Corollary}
\newtheorem{question}[equation]{Question}
\newtheorem*{thm*}{Theorem}
\newtheorem*{lemma*}{Lemma}
\newtheorem*{cor*}{Corollary}

\theoremstyle{remark}

\newtheorem{remark}[equation]{Remark}

\newtheorem{notation}[equation]{Notation}
\newtheorem{terminology}[equation]{Terminology}

\newtheorem*{remark*}{Remark}

\begin{document}

\title[Gorenstein homological algebra and UCT's]{Gorenstein homological algebra and universal coefficient theorems}

\author{Ivo Dell'Ambrogio}
\author{Greg Stevenson}
\author{Jan \v S\v tov\'i\v cek}

\thanks{The first named author is partially supported by the Labex CEMPI (ANR-11-LABX-0007-01).
The second named author is grateful to the Alexander von Humboldt Stiftung for their support. The third named author was supported by grant GA\v{C}R P201/12/G028 from the Czech Science Foundation.}
%46M99, % methods of category theory in functional analysis: none of the above, but in this section
%46L05, % general theory of C*-algebras
%55U35, % abstract and axiomatic homotopy theory
%55U40  % topological categories, foundations of homotopy theory

\address{Greg Stevenson, Universit\"at Bielefeld, Fakult\"at f\"ur Mathematik, BIREP Gruppe, Postfach 10 01 31, 33501 Bielefeld, Germany}
\email{gstevens@math.uni-bielefeld.de}

\address{Ivo Dell’Ambrogio, Laboratoire de Math\'ematiques Paul Painlev\'e, Universit\'e de Lille 1, Cit\'e Scientifique – B\^at. M2, 59665 Villeneuve-d'Ascq Cedex, France}
\email{ivo.dellambrogio@math.univ-lille1.fr}

\address{Jan \v S\v tov\'i\v cek, Department of Algebra, Charles University in Prague, Sokolovsk\'a~83, 186 75 Praha~8, Czech Republic}
\email{stovicek@karlin.mff.cuni.cz}

 \date{\today}

\begin{abstract}
We study criteria for a ring -- or more generally, for a small category -- to be Gorenstein and for a module over it to be of finite projective dimension. 
The goal is to unify the universal coefficient theorems found in the literature and to develop a machinery for proving new ones. 

Among the universal coefficient theorems covered by our methods we find, besides all the classic examples, several exotic examples arising from the KK-theory of C*-algebras and also Neeman's Brown-Adams representability theorem  for compactly generated categories.
\end{abstract}

%\keywords{}

\maketitle

\setcounter{tocdepth}{1} % sets the "depth" of the table of contents
\tableofcontents

\section{Introduction}

Let $\cT$ be a triangulated category, let $\cA$ be an abelian category, and let $h\colon \cT\to \cA$ be a homological functor. 
Suppose we wish to use $h$ in order to approximate the triangulated category $\cT$ by the abelian category~$\cA$.
Generally speaking, it is possible to consider cellular or similar decompositions of the objects of~$\cT$ in order to set up Adams-type spectral sequences; see e.g.\ \cite{meyer_hom} for a very broad-ranging formulation of this idea. 
On the other hand, $h$ could already be an equivalence; but then $\cA$ is necessarily semi-simple, so this case is not so interesting.
The next simplest thing that could happen is what is often called a ``Universal Coefficient Theorem'', or ``UCT'' for short, and will be the focus of this article. In this case there are certain canonical short exact sequences
\begin{align} \label{eq:UCT_intro}
\xymatrix@1@C=16pt{
0 \ar[r] & 
 \Ext^1_\cA(h\Sigma X, h Y) \ar[r] & 
 \Hom_\cT(X,Y) \ar[r] &
  \Hom_\cA(hX, hY) \ar[r] &
   0 
}
\end{align}
determining the Hom groups of $\cT$ up to a one-step extension problem (see Def.\,\ref{defi:uct}).

Typically, this happens when the abelian category $\cA$ has global dimension one, in which case~\eqref{eq:UCT_intro} can be seen as a result of the early collapse of the aforementioned spectral sequence.
The prototypical example is when $\cT=\D(R)$ is the derived category of a hereditary ring and $h=H^*\colon \cT \to \GrMod R=:\cA$ is the functor taking a complex to its graded homology module.
The versions for graded or dg~algebras are also classical (see Example~\ref{ex:classical_uct}), and are the source of the more basic UCT and K\"unneth exact sequences in the topologist's arsenal.

There are also situations where the UCT exact sequence \eqref{eq:UCT_intro} only holds for specific pairs  $(X,Y)$ of objects of~$\cT$. 
One such result, widely used by operator algebraists and noncommutative geometers, is the UCT of Rosenberg and Schochet~\cite{rs} (see Example~\ref{ex:classical_uct_kk}). It applies to $\cT =\KK$, the Kasparov category of separable complex C*-algebras, where now $\cA $ is the category of $\Z/2$-graded abelian groups and $h=K_*$ is the topological K-theory functor. In this example $X$ and $Y$ are two C*-algebras, and for the UCT to hold $X$ must belong to the so-called Bootstrap class, that is, the localizing subcategory of $\KK$ generated by~$\C$.

Interestingly, there are also examples of UCTs which hold for all $X$ in some suitably nice subcategory, where however the abelian category $\cA$ is \emph{not} of global dimension one. What happens in these cases is that although there exist objects in $\cA$ of infinite projective dimension, the homological functor $h$ takes values only in objects of projective dimension one -- which suffices to establish the UCT. Such ``exotic'' examples have arisen in recent years in the work of Ralf Meyer and collaborators on certain variants of KK-theory (see the references in~\S\ref{sec:KK}). 
Our goal is to build on their insights in order to provide a unified conceptual framework for understanding and proving such UCTs.

We make throughout the following hypothesis, which seems to cover all naturally occurring situations:
\begin{hyp}
The homological functor $h$ is the restricted Yoneda functor induced by a small full suspension closed subcategory $\cC\subset \cT$. 
That is to say, $\mathcal A$ is the abelian category $\Mod \cC$ of right $\mathcal C$-modules (i.e.\ contravariant additive functors on~$\cC$) and $h$ is the functor $h_\cC \colon \cT \to \Mod \cC$ sending an object $X$ to the right $\cC$-module  $C\mapsto \cT(C,X)$.
\end{hyp}

For instance, for $\cT=\D(R)$ or $\cT = \KK$ we would choose $\cC$ to consists of all the suspensions of $R$ or $\C$, respectively. The ``exotic'' examples are more involved. 

It follows from standard homological algebra that, in order to establish a UCT exact sequence for a pair $(X,Y)$ of objects in~$\cT$, it suffices that $X$ belongs to the localizing subcategory generated by~$\cC$ and that the $\cC$-module $h_\cC X$ has projective dimension one (see Theorem~\ref{thm:uct} for details). 
This leads us to:

\begin{question} \label{question:general}
How can we recognize those full subcategories $\cC\subset \cT$ for which we have $\pdim_\cC X\leq 1$ for all $X \in \Loc(\cC)$, and therefore a UCT exact sequence \ref{eq:UCT_intro} for all pairs of objects $(X,Y)$ with $X\in \Loc(\cC)$? 
\end{question}

As it turns out, in all the examples where this works that we are aware of, the category $\cC$ is \emph{1-Gorenstein} (see \S\ref{sec:recoll_Gor}). This entails for instance the observed dichotomy that $\cC$-modules have either projective dimension at most one, or are of infinite projective dimension. This observation is quite useful, because Gorenstein rings and categories are rather well-behaved and Gorenstein homological algebra is already a well-developed subject.

With this in mind,  we can now break down Question~\ref{question:general} into the two following more manageable subquestions:

\begin{question} \label{question:1}
How can we tell when a given small category $\cC$ is 1-Gorenstein?
\end{question}

\begin{question} \label{question:2}
Given a 1-Gorenstein subcategory $\cC\subset \cT$, how do we know whether the homological functor $h_\cC\colon \cT\to \Mod \cC$ only takes values in $\cC$-modules of finite (hence $\leq1$) projective dimension? 
\end{question}

Note that Question~\ref{question:1} is independent of the ambient triangulated category. 
On the contrary, Question~\ref{question:2}, as we shall see, crucially depends both on the embedding $\cC\subset \cT$ and on the triangulation of~$\cT$. Moreover, our answers to both questions will turn out to work uniformly for $n$-Gorenstein categories for all $n\geq0$ ($n$ being the Gorenstein dimension of~$\cC$, which is equal to both the maximum finite projective and the maximum finite injective dimension that $\cC$-modules can have).

\begin{center} *** \end{center}

Let us outline the main results and sketch the organization of the paper.

The first part of the article (\S\ref{sec:GTP}-\ref{sec:triangulated_few}) is dedicated to answering Question~\ref{question:1}.
After some theoretical setup, we succeed in proving several criteria for recognizing Gorenstein small categories, the more concrete of which can be applied to descriptions of $\cC$ by generators and relations, i.e., by ``looking at~$\cC$''. Our general philosophy is that a Gorenstein category is one which is, in some sense, sufficiently symmetric over some base Gorenstein ring or category.

We prove for instance: 

\begin{thm}[See Theorem~\ref{thm:Serre_Gor}]
Let $\cC$ be a small category such that:
\begin{enumerate}
\item $\cC$ is bounded, that is, for any fixed object $C\in \cC$ there are only finitely many objects of $\cC$ mapping nontrivially into or out of~$C$.
\item $\cC$ is an $R$-category, with $R$ an $n$-Gorenstein ring (e.g.\ $R=\Z$ and $n=1$).
\item The Hom $R$-modules $\Hom_\cC(C,D)$ are all finitely generated projective.
\item Each unit map $R\to \End_\cC(C)$, $C\in \cC$, admits an $R$-linear retraction.
\item $\cC$ admits a Serre functor relative to~$R$, that is, a self-equivalence $S\colon \cC\stackrel{\sim}{\to}\cC$ equipped with a natural isomorphism 
\[
\Hom_\cC(C,D) \cong \Hom_R(\Hom_\cC(D, SC), R) \,.
\]
\end{enumerate}
Then $\cC$ is $n$-Gorenstein.
\end{thm}

This criterion can be easily applied to obtain new proofs of well-known examples of Gorenstein rings and categories, for instance group algebras and groupoid categories (Example~\ref{ex:groupalgebras}), or categories of chain complexes (Example~\ref{ex:complexes}).

We also prove a more complicated version of the above result where, as it were, the Gorenstein base $R$ is allowed to change from object to object; see Theorem~\ref{thm:Gorenstein}.
This will be used to show that a certain exotic UCT for the equivariant KK-theory of C*-algebras also arises from a 1-Gorenstein category; see Theorem~\ref{thm:eq-KK-Gore}.

We also prove the following criterion, which generalizes results of Neeman and Beligiannis, and is of a rather different spirit:

\begin{thm}[See Theorem~\ref{thm:triangulated_n_Gorenstein}] \label{thm:triangulated_n_Gorenstein_intro}
Let $\cC$ be a triangulated category which admits a skeleton with at most $\aleph_n$ morphisms.
Then $\cC$ is $m$-Gorenstein for some $m\leq n+1$.
\end{thm} 

The second part of our article (\S\ref{sec:exact_vs_fd}-\ref{sec:gore_closed}) addresses Question~\ref{question:2}.
This time our criteria are less concrete, and thus harder to apply in practice, but on the other hand they provide a fully satisfying conceptual answer. Namely, what the next theorem says is that the answer to Question~\ref{question:2} is positive precisely when the singularity category of $\cC$ (see Remark~\ref{rem:sing}) can be generated by syzygies of distinguished triangles contained in $\add \cC$, the additive hull of~$\cC$ inside~$\cT$. The definition of a Gorenstein projective module is recalled in~\S\ref{sec:recoll_Gor}.
(The local coherence hypothesis in the theorem is a very mild one and is satisfied by all the examples we have.)

\begin{thm}[See Theorem~\ref{thm:exact_detection} and Proposition~\ref{prop:equiv_Gore_closed}]
Let $\cC$ be a small suspension closed full subcategory of a triangulated category~$\cT$, and assume that  $\cC$ is Gorenstein and that $\Mod \cC$ is locally coherent. Then the following are equivalent:
\begin{enumerate}
\item The $\cC$-module $h_\cC X$ has finite projective dimension for every object $X\in \cT$.
\item \label{it:syzygy} For every finitely presented Gorenstein projective $\cC$-module~$M$, there exists a distinguished triangle $X\to Y\to Z\to \Sigma X$ in $\add \cC$ such that $M\cong \Img(h_\cC X\to h_\cC Y)$.
\item There exists a set $\mathcal S$ of finitely presented Gorenstein projective $\cC$-modules such that:
\begin{itemize}
\item Every $M\in \mathcal S$ occurs as a syzygy of a triangle in $\add \cC$, as in~\eqref{it:syzygy}.
\item The modules of $\mathcal S$, together with the finitely presented projectives, generate all finitely presented Gorenstein projectives by extensions and retracts.
\end{itemize}
\item If $X\to Y$ is a morphism in $\add \cC$ such that the image of $h_\cC X\to h_\cC Y$ is Gorenstein projective, then the cone of $X\to Y$ is also in $\add \cC$.
\end{enumerate}
\end{thm}

If the equivalent conditions of the theorem are satisfied we say that $\cC$ is \emph{Gorenstein closed} in~$\cT$.
It is now easy to combine the above results with general homological algebra in order to obtain the following new abstract form of the UCT:

\begin{thm}[See Theorem~\ref{thm:gore_uct}] \label{thm:main_simplified}
Let $\cT$ be a triangulated category admitting arbitrary set-indexed coproducts. 
Let $\cC\subset \cT$ be a small suspension closed full subcategory of compact objects, and assume that $\cC$ is 1-Gorenstein and Gorenstein closed in~$\cT$, as above.
Then the UCT sequence \eqref{eq:UCT_intro} with respect to $\cC$ is exact for all $X\in \Loc(\cC)$. 

Moreover, we have the following dichotomy for an arbitrary $\cC$-module~$M$:
\begin{itemize}
\item
either $\pdim_\cC M \leq 1$ and $M\cong h_\cC X$ for some $X\in \Loc(\cC)\subseteq \cT$,
\item
or $\pdim_\cC M=\infty$ and $M$ is not of the form $h_\cC X$ for any $X\in\cT$.
\end{itemize}
\end{thm}

To see how powerful this form of the UCT really is, choose $\cC$ to be the sucategory of all compact objects in~$\cT$, and combine it with Theorem~\ref{thm:triangulated_n_Gorenstein_intro}; the result is Neeman's general form of the Brown-Adams representability theorem~\cite{neeman:adams_brown}: 

\begin{cor}[See Theorem~\ref{thm:brown-adams} for details]
Let $\mathcal T$ be a compactly generated triangulated category whose subcategory of compact objects, $\cT^c$, admits a countable skeleton.
Then every cohomological functor on $\cT^c$ is representable by an object of~$\cT$ and every natural transformation between such functors is induced by a map in~$\cT$.
\end{cor}

Note that $\cT^c$ is locally coherent but hardly ever locally noetherian. Thus in order to deduce Brown-Adams representability from our UCT, as above, it is crucial that we do not make any unnecessary noetherian hypotheses in our theory. 

We also point out that Theorem~\ref{thm:main_simplified} is actually only a special case of Theorem~\ref{thm:gore_uct}. Indeed, the latter applies also to triangulated categories that do not necessarily have all small coproducts; for instance, it has a variant that works for small idempotent complete triangulated categories, or for categories with only countable coproducts. This is necessary if we want to cover the examples from KK-theory, which typically yields categories of the latter kind. 
In order to uniformly treat all possibilities, throughout Section~\S\ref{sec:UCT} we consider triangulated categories admitting all coproducts indexed by sets whose size is bounded by a fixed cardinal~$\aleph$. 
As this complicates things a little, we carefully reprove in this setting all the homological algebra we need. 

As a by-product of our cardinal tracking, we also prove a countable variant of Brown-Adams representability which seems to have remained unrecorded so far; see Theorem~\ref{thm:countable_brown-adams}.

Finally, Section~\S\ref{sec:KK} applies the above results to the motivating exotic examples of UCTs discovered in the realm of KK-theory. More precisely we show how to obtain new proofs, as corollaries of (the countable version of) Theorem~\ref{thm:main_simplified}, for two of these theorems (see Theorems~\ref{thm:umct_gore} and~\ref{thm:filt-KK-Gore-closed}); moreover, in one case where we (still) don't have a complete new proof, we nonetheless prove that the UCT in question fits perfectly into our framework (see \ref{thm:eq-KK-Gore}). 

The last result of our paper, Theorem~\ref{thm:KK_brown_adams}, applies the above-mentioned countable version of Brown-Adams representability to the KK-theoretic setting.

\subsection{Conventions and notations} \label{subsec:conventions}
We will work over a fixed base commutative ring~$\base$, which in the examples will usually be a field or the ring of integers~$\Z$.
Thus, unless otherwise stated, all categories are $\base$-categories and all functors are assumed to be $\base$-linear; the symbol $\otimes$ will stand for the tensor product $\otimes_\base$ over~$\base$, $\Fun(\cC,\cD)$ for the $\base$-category of $\base$-linear functors $\cC\to \cD$, and so on. 

If $\mathcal C$ is a small $\base$-category and $c,d\in \mathcal C$ are two objects, we will often write ${}_d\mathcal C_c$ for the Hom set $\Hom_\mathcal C(c,d)$ considered as a bimodule over the rings ${}_d\mathcal C_d=\End_\mathcal C(d)$ and ${}_c\mathcal C_c=\End_\mathcal C(c)$.
We will write $\Mod \cC := \Fun (\mathcal C^\op, \Mod \base) $ for the $\base$-category of right $\cC$-modules, i.e., of $\base$-linear functors $\cC^\op\to \Mod \base$, and $\modu \cC$ for the full subcategory of finitely presentable $\cC$-modules. 
A $\cC$-module is \emph{free} if it is a coproduct of representables.
In our notation, the Yoneda embedding $\mathcal \cC \to \Mod \cC$ is the fully faithful $\base$-linear functor sending an object $c\in \cC$ to the right $\cC$-module ${}_c\cC\colon d\mapsto {}_c\cC_d$.

Recall that the Yoneda embedding $\cC\to \proj \cC$ induces an equivalence between $\Mod(\proj \cC)$ and $ \Mod \cC$ (see e.g.\ \cite{auslander:rep1}*{Prop. 2.5}). Using this, we will often evaluate $\cC$-modules on finitely generated projectives, by a slight abuse of notation. 

Some (rather straightforward) conventions involving cardinal numbers will be introduced in Terminology~\ref{ter:cardinals}.

\subsection*{Acknowledgements}
We are very grateful to Ralf Meyer for asking us the interesting questions that have prompted our collaboration. Thanks are also due to him for first noticing that the Gorenstein closed condition of Theorem~\ref{thm:exact_detection} is not only sufficient but also necessary.
Another source of inspiration was a talk given in Bielefeld by Claus Ringel where he presented the results of~\cite{ringel-zhang:dual}  (see Example~\ref{ex:ringel-zhang}).

\section{Preliminaries on Gorenstein categories}
\label{sec:recoll_Gor}

If $A$ is an object of an abelian category~$\cA$, denote by $\pdim A$ and $\idim A$ its projective and injective dimension, respectively (both can be defined in terms of vanishing of Ext groups even if $\cA$ does not have enough projective or injective objects).
The \emph{finitary projective dimension} of $\mathcal A$ is by definition $\sup\{\pdim A \mid A\in \cA \textrm{ with } \pdim A< \infty \}$. The \emph{finitary injective dimension} of $\cA$ is defined similarly.

\begin{defi} [\cite{enochs-etal:Gorenstein_cats}*{Def.\,2.18}] \label{defi:Gorenstein}
A \emph{Gorenstein category} is a Grothendieck category $\cA$ satisfying the following three additional conditions:
\begin{enumerate}
\item $\pdim A<\infty $ if and only if $\idim A<\infty$ for every object $A\in \cA$;
\item the finitary injective and projective dimensions of $\cA$ are finite, and
\item $\cA$ has a generator with finite projective dimension.
\end{enumerate}
If $\cC$ is a small $\base$-category, we will say $\cC$ is \emph{Gorenstein}  if its category of representations, $\cA=\Mod \cC$, is Gorenstein in the above sense. All our examples will be of this form. %and observe that in this case condition (3) is always satisfied since $\Mod \cC$ has a set of finite projective generators, namely, the representable modules.
\end{defi}

The definition of a small Gorenstein category $\cC$ applies in particular to the case when $\cC$  has only one object, i.e.\ when $\cC$ is a unital, associative, (not necessarily commutative) $\base$-algebra~$R$. 

\begin{example} \label{ex:iwanaga-gore}
A ring $R$ is said to be \emph{Iwanaga-Gorenstein} (\cite{iwanaga:gorenstein2},~\cite{enochs-jenda:gorenstein_book}) if it is left and right noetherian and it has finite self-injective dimension both as left and right module. It follows that $\idim {}_RR = \idim R_R$.
If $R$ is Iwanaga-Gorenstein then it is Gorenstein in the sense of this article, i.e.\ the abelian category $\Mod R$ is Gorenstein (and moreover locally noetherian). The converse is not necessarily true, not even if the ring is right noetherian. Indeed, \cite{jategaonkar:counterex} constructs for every $1< n\leq \infty$ a left noetherian, left hereditary ring $R$ with right global dimension~$n$. It follows that $\Mod R^\op$ is a locally noetherian Gorenstein category, but the ring $R^\op$ is not left noetherian (if it were, its left and right global dimensions would have to be equal as both would coincide with its Tor-dimension). 

However, if $R$ is two-sided noetherian, then $\Mod R$ being Gorenstein in the sense of Def.\,\ref{defi:Gorenstein} is the same as $R$ being Iwanaga-Gorenstein. One only needs to check that $\idim {}_RR < \infty$, but this follows from the fact that the injective right $R$-module $\Hom_\Z({}_RR,\Q/\Z)$ has finite projective dimension. Indeed, we have that $0 = \Tor^R_i(\Hom_\Z({}_RR,\Q/\Z),R/I) \cong \Hom_\Z(\Ext^i_R(R/I,{}_RR),\Q/\Z)$ for some fixed $i \gg 0$ and an arbitrary left ideal $I \subseteq {}_RR$ by~\cite{enochs-jenda:gorenstein_book}*{Theorem 3.2.13}.
\end{example}

Let $\cA$ be a Gorenstein category.
An object $M\in \cA$ is \emph{Gorenstein projective} if it admits a complete projective resolution and is \emph{Gorenstein injective} if it admits a complete injective resolution (see \cite{enochs-etal:Gorenstein_cats}*{Def.\ 2.20} for details).
Define the \emph{Gorenstein projective dimension} of $M\in \cA$ to be~$n$, in symbols $\Gpdim M = n$, if the first syzygy of $M$ that is Gorenstein projective is the $n$-th one (and set $\Gpdim M=\infty$ if there is no such syzygy).  
Note that the Gorenstein projective dimension is also the minimal length of a resolution of $M$ by Gorenstein projectives (see \cite{auslander-bridger}*{Lemma~3.12}, and also \cite{stovicek:derived_big_cotilting}*{Proposition~2.3} for a simpler and more general proof).
The \emph{Gorenstein injective dimension} is defined dually, using cosyzygies and Gorenstein injectives, and one can define the global Gorenstein projective \textup(resp.\ injective\textup) dimension of~$\cA$ by taking the supremum over all $M\in \cA$.

\begin{center} ***\end{center}

For the remainder of this section, assume that the Grothendieck category $\cA$ has enough projectives, e.g.\ $\cA= \Mod \cC$.

\begin{prop} [{\cite{enochs-etal:Gorenstein_cats}*{Theorem 2.28}}] 
\label{prop:dimensions}
The category $\cA$ is Gorenstein \iff its global Gorenstein injective and Gorenstein projective dimensions are both finite.
Moreover, if this is the case then its finitary injective, finitary projective, global Gorenstein injective, and global Gorenstein projective dimensions all coincide.
\end{prop}

\begin{defi} \label{defi:Gore_dim}
The integer $n\in \N$ in the above proposition is simply called the \emph{Gorenstein dimension} of~$\cA$, and we say for short that $\cA$ is \emph{$n$-Gorenstein}. 
In this case the Gorenstein projectives are precisely the objects in $\cA$ which are $n$-th syzygies (see \cite{enochs-etal:Gorenstein_cats}*{Theorem~2.26}). 
\end{defi}

We will use the following criterion for recognizing Gorenstein categories:

\begin{lemma} \label{lemma:Gore_criterion}
Let $\cA$ be a Grothendieck category with enough projectives. Then $\cA$ is Gorenstein of Gorenstein dimension at most~$n$, provided the following holds: 
every projective object of~$\cA$ has an injective resolution of length at most~$n$ and every injective object has a projective resolution of length at most~$n$.
\end{lemma}

\begin{proof}
Property (3) in Definition~\ref{defi:Gorenstein} holds because by hypothesis $\cA$ has a projective generator.
With enough projectives and injectives we can now reason by dimension shifting;
it follows easily from the hypothesis on projectives that every object \emph{of finite projective dimension} has injective dimension at most~$n$, and from the hypothesis on injectives that every object \emph{of finite injective dimension} has projective dimension at most~$n$. This proves that $\cA$ satisfies condition~(1). But then (2) must hold by symmetry:  if an object has finite projective (resp.\ injective) dimension, then by the above remarks it must have injective (resp.\ projective) dimension at most $n$ and therefore also projective (resp.\ injective) dimension at most~$n$.
\end{proof}

\begin{notation} 
We shall consider the following full subcategories: 
\begin{align*}
\Proj \cA &= \{ M\in \cA \mid M \textrm{ is projective }\} \\
\FD \cA &= \{M\in \cA \mid \pdim M < \infty \} \\
\GProj \cA &= \{G\in \cA \mid G \; \text{is Gorenstein projective}\} \\
\GInj \cA &= \{G\in \cA \mid G \; \text{is Gorenstein injective}\}
\end{align*}
If $\cA=\Mod \cC$, we will often abuse notation, as one usually does with rings, writing $\GProj \cC$ for $\GProj \cA$, mentioning the Gorenstein dimension \emph{of~$\cC$}, etc.
\end{notation}

A \emph{complete hereditary cotorsion pair} in $\cA$ is a pair $(\cL, \cR)$ of full subcategories with the following three properties:
\begin{enumerate}
\item $\cL= \{X\mid \Ext^1_\cC(X,R)=0 \, \forall R\in \cR\}$ and $\cR= \{Y\mid \Ext^1_\cC(L,Y)=0 \, \forall L\in \cL\}$;
\item $\Ext^n_\cC(L,R)=0$ for all $L\in \cL$, $R\in \cR$ and $n\geq1$;
\item for each $M\in \cA$ there exist exact sequences $0\to R\to L\to M \to 0$ and $0\to M\to R'\to L'\to 0$ with $L,L'\in \cL$ and $R,R'\in \cR$.
\end{enumerate}

\begin{remark} \label{rem:cotors}
Property (1) says that $(\cL,\cR)$ is a cotorsion pair, and (2) and (3) say that it is hereditary and complete, respectively.
If (2) and (3) hold, then (1) follows from the weaker assumption that both $\cL$ and $\cR$ are closed under retracts.
\end{remark}

\begin{prop} [{\cite{enochs-etal:Gorenstein_cats}*{Theorems 2.25 and 2.26}}]  \label{prop:pairs}
If $\cA$ is Gorenstein, we have in $\cA$ the complete hereditary cotorsion pairs $(\FD \cA, \GInj \cA)$ and $(\GProj \cA, \FD \cA)$.
\end{prop}

Note that both $\GProj \cA$ and $\GInj \cA$ are extension closed in $\Mod \cA$ and therefore both inherit an exact structure (we refer to \cite{buehler:exact} for exact categories). In general we will make statements only for one of $\GProj \cA$ or $\GInj \cA$ and it is understood that the dual holds.

\begin{prop} \label{prop:triangulated}
Let $\cA$ be a Gorenstein category. Then $\GProj \cA$ is a Frobenius exact category, with $\Proj \cA$ as its subcategory of injective-projectives. In particular, the (additive) quotient category $\stGProj \cA:= \GProj \cA / \Proj \cA$ inherits a canonical structure of triangulated category.
\end{prop}

\begin{proof} 
It is a standard fact, proved by Happel in \cite{happel:triangulated}*{Ch.\,I}, that the stable category of a Frobenius category is triangulated. That $\GProj \cA$ is Frobenius is also standard, and an easy consequence of Proposition~\ref{prop:pairs}.
\end{proof}

In fact, Proposition~\ref{prop:pairs} allows for another point of view on the triangulated structure which will be useful below, using a more homotopy theoretic language. We refer to~\cites{hirschhorn,hovey:model} for the relevant terminology and basic facts.

\begin{prop} \label{prop:models}
Let $\cA$ be a Gorenstein category. Then there are two stable model structures on $\cA$ whose homotopy categories are triangle equivalent to $\stGProj \cA$. In both cases, the class of weak equivalences consists of all compositions $w_ew_m$, where $w_m$ is a monomorphism with cokernel in $\FD \cA$ and $w_e$ is an epimorphism with kernel in $\FD \cA$.

The first model structure (the \emph{Gorenstein projective model structure}) has mono\-mor\-phisms with Gorenstein projective cokernels as cofibrations and all epimorphisms as fibrations. The second model structure (the \emph{Gorenstein injective model structure}) has all monomorphisms as cofibrations and epimorphisms with Gorenstein injective kernels as fibrations.
\end{prop}

\begin{proof}
The idea comes from \cite{hovey:abelian-models}*{\S8}, although there the theorem is proved only for $\cA = \Mod R$ with $R$ Iwanaga-Gorenstein. All that was needed, however, were the cotorsion pairs from Proposition~\ref{prop:pairs} along with the easily verifiable facts that $\GProj\cA \cap \FD\cA = \Proj\cA$ and dually $\GInj\cA \cap \FD\cA$ is precisely the class of injective objects in $\cA$. Then \cite{hovey:abelian-models}*{Theorem 2.2} applies. See also~\cite{stovicek:abelian-models} and references there for an extensive discussion of abelian and exact model structures, and in particular \cite{stovicek:abelian-models}*{\S6} for the above description of weak equivalences.
\end{proof}

\begin{center} ***\end{center}

Now assume that $\cA$ is a Gorenstein category of the form $\Mod \cC$ which moreover is locally coherent. That is, we require that the category $\modu \cC$ of finitely presented modules be abelian, or equivalently that the additive closure of $\cC$ have weak kernels; see~\cite{cb:loc-fin-pres}*{\S2.4}.
(This holds for instance if $\Mod \cC$ is locally noetherian, e.g.\ if $\cC$ is an Iwanaga-Gorenstein ring.)
Here we also want to  consider the following full subcategories:
\begin{align*}
\proj \cC &= \modu \cC \cap \Proj \cC \\
\fd \cC &= \modu \cC \cap \FD \cC  \\
\Gproj \cC &= \modu \cC \cap \GProj \cC
\end{align*}

\begin{thm} \label{thm:triangulated_small}
Assume that $\cA=\Mod \cC$ is Gorenstein locally coherent. Then:
\begin{enumerate}
\item[(i)]
$(\Gproj \cC, \fd \cC)$ is a complete hereditary cotorsion pair in $\modu \cC$.
\item[(ii)]
The exact category $\Gproj \cC$ is Frobenius with $\proj \cC$ as projective-injectives. In particular, the stable category $\underline{\Gproj} \,\cC = \Gproj \cC / \proj \cC$ is triangulated. 
\item[(iii)]
Moreover, we can identify $\underline{\Gproj} \,\cC$ with the full subcategory of compact objects in $\underline{\GProj}\, \cC$.
\end{enumerate}
\end{thm}

\begin{proof}
(i) We have to prove conditions (1)--(3) of the definition. 
Property (2) holds because  by Proposition~\ref{prop:pairs} it holds for all (possibly big) modules.
Since both classes are closed under retracts, by Remark \ref{rem:cotors} it remains only to prove property~(3).
Now property~(3) follows from \cite{auslander-buchweitz}*{Theorem~1.1} by setting 
$\boldsymbol\omega := \proj \cC$ and $\boldsymbol{X}:= \Gproj \cC$, so that $\hat{\boldsymbol{\omega}}= \fd \cC$ and $\hat{\boldsymbol{X}}= \modu \cC$. (Indeed, for the last equality: for any $M\in \modu \cC$ we can construct a resolution by objects in $\proj \cC$; since we are working in the Gorenstein category $\Mod \cC$, at some point a Gorenstein projective syzygy will appear, which will necessarily be finitely presented.)

Part (ii) follows from (i) by the same argument as in the proof of Proposition~\ref{prop:triangulated}.

To prove~(iii), let $W$ be the class of weak equivalences from Proposition~\ref{prop:models} and denote as usual by $\Ho\cA := \cA[W\inv]$ the corresponding homotopy category. We already know that $\Ho\cA \simeq \stGProj \cA$ is triangulated.

Note that each suspension or desuspension of a finitely presented $\cC$-module is again isomorphic to a finitely presented $\cC$-module in $\Ho\cA$. To see this, we employ the Gorenstein projective model structure. Every $G \in \modu\cC$ is by (i) weakly equivalent to some $G' \in \Gproj\cC$. A desuspension of $G'$ is just a syzygy while to construct a suspension we again use the completeness of $(\Gproj \cC, \fd \cC)$. At any rate, both $\Sigma G'$ and $\Sigma\inv G'$ can be taken in $\Gproj \cC$ up to isomorphism.

Next we claim that if $X \in \Ho\cA$ is such that $\Ho\cA(G,X) = 0$ for each $G \in \modu\cC$, then $X \cong 0$ in $\Ho\cA$. Now we use the Gorenstein injective model structure on $\cA$ and assume that $X$ is fibrant there. In other words, $X\in \GInj \cA$ and as such, $X$ is isomorphic in $\cA$ to the zero cocycles $Z^0(I)$ of some acyclic complex of $\cC$-modules
\[ I\colon \quad \cdots \to I^{-2} \to I^{-1} \to I^0 \to I^1 \to I^2 \to \cdots \]
with injective components. In particular, all the other cocycles of $I$ are also Gorenstein injectives and are just (de)suspensions of $X$ in $\Ho\cA$ up to isomorphism. Hence we also have $\Ho\cA(G,Z^n(I)) \cong \K(\cA)(G,\Sigma^n I) = 0$ for each $G \in \modu\cC$. At this point we use results from~\cite{stovicek:pure-derived}.
The complex $I$ is, by the above, a pure acyclic complex in the sense of \cite{stovicek:pure-derived}*{Def. 4.12}, and hence contractible by \cite{stovicek:pure-derived}*{Corollary 5.5} (similar results to~\cite{stovicek:pure-derived} have been also obtained in~\cite{krause:auslander-formula} by different arguments). Thus, $X \cong Z^0(I)$ is injective and so weakly equivalent to $0$, finishing the proof of the claim.

Turning back to the Gorenstein projective model structure, we have proved that if $\underline{\Hom}_\cC(G,X) = 0$ in $\stGProj \cC$ for each $G \in \Gproj\cC$, then $X \cong 0$ in $\stGProj \cC$. As one also easily checks, the inclusion $\Gproj \cC \to \GProj \cC $ descends to a fully faithful embedding $\underline{\Gproj} \,\cC \to \underline{\GProj}\,\cC$ and $\underline{\Gproj} \,\cC$ consists of compact objects in $\stGProj\cC$. The conclusion follows from standard facts about compactly generated triangulated categories; see~\cite{neemanLoc}*{Lemma 2.2} or~\cite{krause:chicago}*{Lemma 6.5}.
\end{proof}

\begin{remark}
The proof of Theorem~\ref{thm:triangulated_small}(iii) is considerably easier if $\cA = \Mod\cC$ is assumed to be locally noetherian. In such a case it is easy to prove by simple dimension shifting that $\FD\cC = \ker \Ext^1_\cC(\Gproj\,\cC,-)$, and consequently
\[ \ker\underline{\Hom}_\cC(\Gproj\,\cC,-) = \FD\cC \cap \GProj\cC = \Proj\cC \textrm{ in } \GProj\cC. \]
\end{remark}

\begin{remark}
\label{rem:sing}
Analogously to the case of Iwanaga-Gorenstein rings treated in \cite{buchweitz}, one obtains a triangle equivalence $\underline{\Gproj} \,\cC \simeq \Db(\modu\cC)/\Db(\proj\cC)$ whenever $\Mod\cC$ is Gorenstein and locally coherent. The latter quotient is known under the names \emph{stable derived category}~\cite{krause:stable} or \emph{singularity category}~\cite{orlov:graded} of $\cC$ and denoted by $\D_{\mathrm{Sg}}(\cC)$. 
In the language of model categories, one first notes that there is a Quillen equivalence $S^0\colon \Mod\cC \rightleftarrows \Ch(\cC)\colon Z^0$, where $S^0$ reinterprets a module as a complex concentrated in degree zero, $Z^0$ is the degree zero cocycle functor, $\Mod\cC$ is equipped with the Gorenstein injective model structure, and $\Ch(\cC)$ carries the model structure from~\cite{becker:models}*{Proposition 2.2.1(2)} (see also \cite{stovicek:pure-derived}*{Proposition 7.6}). The rest quickly follows from \cite{becker:models}*{Corollary 2.2.2} and/or \cite{stovicek:pure-derived}*{Theorem 7.7}.
\end{remark}

\begin{remark}
\label{rem:ding-chen}
For the sake of completeness, we note that there exists an alternative generalization of Iwanaga-Gorenstein rings to the non-noetherian setting studied originally by Ding and Chen~\cites{ding-chen:flat-inj,ding-chen:self-fp-inj}. They considered two-sided coherent rings with finite self-fp-injective dimension on both sides, and modules over such rings. In all our examples of Gorenstein categories $\Mod\cC$, the category $\cC$ (viewed as a ring with several objects) will satisfy these conditions. Hence we probably could, also in view of recent results~\cite{gillespie:ding-chen} and~\cite{stovicek:pure-derived}*{Prop. 7.9}, use this approach too.
\end{remark}

%%%
\section{The Gorenstein transfer property} \label{sec:GTP}

Let $i_*\colon \cA\to \cB$ be a ($\base$-linear) functor between ($\base$-linear) abelian categories, and assume that $i_*$ has both a left adjoint $i^*\colon \cB\to \cA$ and a right adjoint $i^!\colon \cB\to \cA$.
 Recall that the functor $i_*\colon \cA\to \cB$ is \emph{Frobenius} if $i^*$ and $i^!$ are isomorphic functors. Let us consider some familiar special cases involving rings.

\begin{example} \label{ex:algebras}
Let $i\colon B\to A$ be a homomorphism of $\base$-algebras.
Then the restriction functor $i_*\colon \Mod A\to \Mod B$ is Frobenius if and only if $A$ is finitely generated projective as a right $B$-module and ${}_BA_A \cong ({}_AA_B)^\vee:=\Hom_B({}_AA_B,B)$ are isomorphic $B$-$A$-bimodules, that is, if and only if $i$ is a \emph{Frobenius ring extension} in the sense of Kasch~\cite{kasch}.
To see this, simply recall that $i^*= (-)\otimes_B{}_BA_A$ and $i^!=\Hom_B({}_AA_B,-)\cong (-)\otimes_B({}_AA_B)^\vee$, where the last isomorphism holds because $i^* \cong i_!$ implies that  $A$ is a finitely generated projective as a right $B$-module. (This was first observed by K.\,Morita~\cite{morita:adjoints}.)
\end{example}

\begin{example}
In the previous example, assume moreover that $B=\base$ and $\base$ is a field, so that $A$ is a finite dimensional $\base$-algebra.  
Then $i_*\colon \Mod A\to \Mod \base$ is Frobenius precisely if $A$ is a Frobenius algebra in the classical sense. 
\end{example}

\begin{remark}
\label{rem:QF}
The notion of Frobenius functor can be generalized to that of a quasi-Frobenius functor, or even to ``strongly adjoint'' pairs of functors (see~\cite{morita:adjoints}). We require a similar, but slightly different, generalization which is more adapted to the study of Gorenstein categories.
This leads us to the next definition.
\end{remark}

\begin{defi}
\label{defi:GTP}
Let $\cA$ and $\cB$ be two Grothendieck categories with enough projective objects,  
and let $i_*\colon \cA\to \cB$ be a faithful exact functor admitting both a left adjoint $i^*$ and a right adjoint~$i^!$.
We say $i_*$ has the \emph{Gorenstein transfer property}, or \emph{GT-property} for short, if the following holds: 
\begin{itemize}
\item[(GT)] When restricting to the full subcategory $\cF:= \{X\in \cB \mid \pdim X <\infty \} $ of~$\cB$, consisting of objects with finite projective dimension, there is an equivalence $\Phi\colon \cA\stackrel{\sim}{\to}\cA$ such that the functor $i^*|_\cF$ is a retract of a product of copies of~$(\Phi\circ i^!)|_\cF$, and similarly, $(\Phi\circ i^!)|_\cF$ is a retract of a coproduct of copies of~$i^*|_\cF$.
\end{itemize}
We say that $i_*$ has the \emph{strong GT-property}, or \emph{SGT-property}, if (GT) holds because of an isomorphism $i^*|_{\cF}\cong (\Phi\circ i^!)|_\cF$ of functors $\cF\to \cA$.
\begin{equation*}
\xymatrix{
\cA \ar@(dr,ur)[]_{\Phi}
 \ar[d]|{i_*} \\
\cB
 \ar@/^10pt/[u]^{i^*}
 \ar@/_10pt/[u]_{i^!}
}
\end{equation*}
\end{defi}

\begin{lemma} \label{lemma:i_*FD}
Assume that $i_*\colon \cA \to \cB$ has the GT-property and that $\cB$ is Gorenstein.
Then both adjoints $i^*,i^!$ preserve injectives and projectives and restrict to exact functors on $\FD \cB$.
\end{lemma}

\begin{proof}
As in the statement, we suppose $\cB$ is Gorenstein. 
Note that $\cF=\FD \cB$ contains the injectives (projectives) of~$\cB$, as well as their finite projective (injective) resolutions. 
Since $i^*$ has an exact right adjoint, it is right exact and preserves projectives. 
Dually, $i^!$ is left exact and preserves injectives. 
Since the properties of being left exact and preserving injectives are inherited by retracts and products of functors, and are invariant under composition with equivalences, it follows from one half of (GT) that $i^*|_\cF$ is exact and also preserves injectives. The dual argument, using the other half of~(GT), implies that $i^!|_\cF$ is exact and also preserves projectives. 
\end{proof}

\begin{prop}
\label{prop:Gor_Gor_general}
Assume that $i_*\colon \cA\to \cB$ has the GT-property and that $\cB$ is $n$-Gorenstein. Then $\cA$ is $n'$-Gorenstein for some $n'\leq n$.
\end{prop}

\begin{proof}
As a consequence of Lemma \ref{lemma:i_*FD} the functors $i^*$ and $i^!$ send injective resolutions of objects of finite projective dimension to injective resolutions in $\cA$, and similarly for projective resolutions.
We use this to show that in $\cA$ every projective object has an injective resolution of length at most~$n$, and that every injective has a projective resolution of length at most~$n$. 
Since the latter has dual proof, we only show the former.

Let $P \in \Proj \cA$, and choose a projective precover $Q\to i_*P\to 0$ in~$\cB$. Since $i_*$ is faithful the counit of the adjunction $(i^*, i_*)$ is an epimorphism (\cite{maclane}*{Theorem IV.3.1}) and we obtain a composite epimorphism 
$i^*Q \to i^*i_* P \to P$ in~$\cA$, which must split because $P$ is projective in~$\cA$. This shows that $P$ is a retract of a projective of the form $i^*Q$ for some $Q\in \Proj \cB$. Since $Q$ has an injective resolution in $\cB$ of length (at most)~$n$, so does $i^*Q$ and therefore also~$P$. Thus every projective object in $\cA$ has an injective resolution of length at most~$n$.

Since $\cA$ is assumed to have enough projectives, we are done by Lemma~\ref{lemma:Gore_criterion}.
\end{proof}

\begin{prop} \label{prop:i_*GProj}
Assume that $i_*\colon \cA\to \cB$ has the GT-property and that $\cB$ is $n$-Gorenstein.
Then $i_*$ preserves Gorenstein projective objects and Gorenstein injective objects. 

If, moreover, $i_*$ has the SGT-property, then $i_*$ also reflects Gorenstein projectivity and injectivity. That is, $i_*(G) \in \GProj\cB$ if and only if $G \in \GProj\cA$ and similarly for Gorenstein injective objects.
\end{prop}
 
\begin{proof}
We only prove the result for Gorenstein projectives, the proof for Gorenstein injectives being dual. Suppose $G\in \cA$ is Gorenstein projective. By the cotorsion pair $(\GProj \cB, \FD \cB)$, to prove $i_*G$ is Gorenstein projective it is sufficient to check
\begin{displaymath}
\Ext_\cB^\ell (i_*G, Q) = 0 \quad \text{for all} \; Q\in \Proj \cB \;\textrm{and}\; \ell>0 \,.
\end{displaymath}
By Lemma \ref{lemma:i_*FD} the functor $i^!$ is exact on objects of finite projective dimension and preserves both injectives and projectives, hence we can rewrite 
\begin{equation} \label{eq:GP-adj}
\Ext_\cB^\ell (i_*G, Q) \cong \Ext_\cA^\ell (G,i^! Q)  \quad \text{for all} \; G \in \cA, \; Q\in \Proj \cB \;\textrm{and}\; \ell\le0.
\end{equation}
It follows that $\Ext_\cB^\ell (i_*G, Q) = 0$ for all $\ell>0$ by the Gorenstein projectivity of $G$ and the projectivity of $i^!Q$.

Suppose conversely that $i_*$ has the SGT-property, i.e.\ $i^*|_{\FD\cB}\cong (\Phi\circ i^!)|_{\FD\cB}$, and that $i_*G$ is Gorenstein projective. Given any $P \in \Proj \cA$, there exist $Q \in \Proj \cB$ and a retraction $i^*Q \to \Phi P$ by the argument in the proof of Proposition~\ref{prop:Gor_Gor_general}. Since now $P$ is a retract of $i^!Q \cong \Phi\inv i^*Q$, \eqref{eq:GP-adj} implies that $\Ext_\cA^\ell (G,P) = 0$ for all $\ell>0$. Hence $G \in \GProj \cA$.
\end{proof}

Recall from \S\ref{sec:recoll_Gor} that for a Grothendieck category $\cA$ with enough projectives, being $n$-Gorenstein is equivalent to its global Gorenstein injective and projective dimensions both being~$n$.

\begin{prop} \label{prop:equal_Gd}
In the situation of Proposition~\ref{prop:Gor_Gor_general}, assume moreover that each component of the unit of adjunction $M\to i_*i^*M$ is a split mono, or that each component of the counit $i_*i^!M\to M$ of the other adjunction is a split epi (in~$\cB$). Then the Gorenstein dimensions of $\cA$ and $\cB$ agree.
\end{prop}

\begin{proof}
Let $n$ be the global Gorenstein dimension of $\cB$ and $n'$ that of $\cA$, so $n'\leq n$.
Let $M\in \cB$. Assume we are in the first case: $M\to i_*i^* M$ is a split mono (the proof in the other case is similar).
By assumption $i^*M$ has a Gorenstein projective resolution of length at most~$n'$. 
By Proposition~\ref{prop:i_*GProj}, restricting it along the exact functor $i_*$ gives a Gorenstein projective resolution of $i_*i^*M$. Since $M$ is a summand of $i_*i^*M$ by the additional splitting hypothesis, it follows that $\Gpdim M \leq n'$. Thus, by the recollection above and since $M$ is arbitrary, $\cB$ is $n'$-Gorenstein and so, since $n'\leq n \leq n'$, we must have $n=n'$.
\end{proof}

The next theorem is obtained by combining Propositions \ref{prop:Gor_Gor_general} and \ref{prop:equal_Gd}.

\begin{thm}
\label{thm:GorGor_general}
Let $i_*\colon \cA\to \cB$ be a faithful exact functor between Grothendieck categories having the GT-property (Definition~\ref{defi:GTP}), and such that either the unit $\id_\cB\to i_*i^*$ is a componentwise split mono or the counit $i_*i^!\to \id_\cB$ is a componentwise split epi.
Then if $\cB$ is $n$-Gorenstein so is~$\cA$.
\qed
\end{thm}

 \begin{example}
Note that a faithful Frobenius functor between Grothendieck categories certainly has the (strong) Gorenstein transfer property. 
Thus if $i\colon B\to A$ is a Frobenius extension of an $n$-Gorenstein ring~$B$, it follows from Proposition \ref{prop:Gor_Gor_general} applied to $i_*\colon \Mod A\to \Mod B$ that $A$ must be Gorenstein with Gorenstein dimension at most~$n$. In particular, for $B=\base$ a field we recover the classical fact that Frobenius algebras are selfinjective.

More generally, assume there exists a $B$-linear retraction $\rho\colon A\to B$ of the Frobenius extension $i\colon B\to A$ (so that $\rho i =\id_B$). Then we may define a $B$-linear map $\rho_M\colon M\otimes_B A\to M$ for every $B$-module~$M$ by setting $\rho_M(m\otimes a)=m \rho(a)$. Since this is a retraction of the unit of adjunction $M\to i_*i^*M= M\otimes_BA_B$, $m\mapsto m\otimes 1$, in this case $A$ and $B$ must have the same Gorenstein dimension by Proposition~\ref{prop:equal_Gd}.

\end{example}

We now turn to a new class of examples.

%%%
\section{Serre functors and the GT-property}
\label{sec:Serre}

Let $\cC$ be a small $\base$-category.

\begin{defi}\label{defi:Hom-finite}
We will say $\cC$ is \emph{Hom-finite} if ${}_d\cC_c=\Hom_\cC(c,d)$ is a finite projective $\base$-module for all objects $c$ and~$d$. 
A \emph{Serre functor on $\cC$ \textup(relative to~$\base$\textup)} is an auto-equivalence $ S\colon \cC\stackrel{\sim}{\to} \cC$ together with an isomorphism 
\[
\sigma_{c,d}\colon \Hom_\cC(c,d)\stackrel{\sim}{\longrightarrow} \Hom_\cC(d, Sc)^*
\]
natural in $c,d\in \cC$. Here $(-)^*=\Hom_\base(-,\base)$ denotes the $\base$-dual.
\end{defi}

\begin{remark}
Beware that other notions of ``Hom-finite'' are in use for categories over a general commutative ring: e.g.\ some authors mean that the Hom-modules are of finite length, or just finitely generated.
\end{remark}

\begin{remark}
Note that if a Serre functor exists, it is unique up to a natural equivalence. This follows from the fact $Sc$ is, for each $c \in \cC$, determined by the isomorphism of $\cC$-modules $(\cC_c)^* \cong {_{Sc}\cC}$ and the Yoneda lemma. Hence the existence of a Serre functor can be viewed as a property of $\cC$.
\end{remark}

\begin{remark} \label{rem:trace-Serre}
It is often convenient to express the isomorphism $\sigma_{c,d}$ from Definition~\ref{defi:Hom-finite} in the form of a collection of $\base$-linear maps
\[
\lambda_c\colon \Hom_\cC(c,Sc) \longrightarrow \base,
\] 
one for each object of $\cC$. Starting with $(\sigma_{c,d})_{c,d \in \cC}$, one puts $\lambda_c(f) = \sigma_{c,c}(1_c)(f)$. Conversely, one can reconstruct $\sigma_{c,d}$ as $\sigma_{c,d}(f)(g) = \lambda_c(gf) = \lambda_d(S(f)g)$. The required naturality of $\sigma_{c,d}$ is equivalent to $(\lambda_c)_{c \in \cC}$ satisfying the trace-like condition
\[
\lambda_c(gf) = \lambda_d(S(f)g) \quad \textrm{for each $f\colon c \to d$ and $g\colon d \to Sc$ in $\cC$.}
\] 
\end{remark}

\begin{defi}
\label{defi:neighbourhood}
We say $\cC$ is \emph{locally bounded} if for every $c\in \cC$ the two neighbourhood categories $\incoming(c)$ and $\outcoming(c)$ have only finitely many objects, where, for every object $c\in \cC$, we define the \emph{incoming neighbourhood of~$c$} and the \emph{outgoing neighbourhood of~$c$} to be the full subcategories
\[
\incoming(c):=\{d\in\cC\mid {}_c\cC_d\neq0\} 
\quad\textrm{ and }\quad
\outcoming(c):=\{d\in\cC\mid {}_d\cC_c\neq0\} \,,
\]
respectively.
\end{defi}

If $R$ is an arbitrary $\base$-algebra, we may consider the extension of $\cC$ by~$R$, that is the $\base$-category $R\otimes \cC$ with the same objects as~$\cC$, Hom modules $(R\otimes \cC)(c,d)=R\otimes \cC(c,d)$ and composition given by $(s\otimes g)\circ (r\otimes f)= sr \otimes gf$.
Recall also that $R$ is an $n$-Gorenstein algebra if $\Mod R$ is an $n$-Gorenstein category in the sense of Definition~\ref{defi:Gorenstein}.

\begin{thm}\label{thm:Serre_Gor}
Let $\cC$ be a locally bounded  Hom-finite $\base$-category equipped with a Serre functor $( S, \sigma)$ relative to~$\base$. 
If $R$ is any $n$-Gorenstein $\base$-algebra, then $R\otimes \cC$ is $n'$-Gorenstein for some $n'\leq n$.
If moreover the unit $R \to  R\otimes {}_c\cC_c$, $1_R\mapsto 1_R\otimes 1_c$, of every object $c$ has a retraction commuting with the right $R$-actions then we must have $n'=n$.
\end{thm}

\begin{proof}
Consider the evaluation functors $\ev_d \colon \cA:=\Mod (R\otimes \cC )\to \Mod R$ for $d\in \cC$.
Setting $\cB:= \prod_{d} \Mod R =\Mod (\coprod_d R)$, all these functors assemble into a faithful functor $i_*\colon \cA\to \cB$ between Grothendieck categories. 
In this situation, the left and right adjoints $i^*$ and $i^!$ of $i_*$ have a simple description:
for a collection of $R$-modules $M=(M_d)_{d\in \obj \cC} \in \cB$ their values at $c\in \cC$ are given, respectively, by 
\[
(i^*M)_c 
= \coprod_d M_d \otimes_R (R\otimes {}_d\cC_c)
= \coprod_d M_d\otimes {}_d\cC_c 
\]
and
\[
(i^!M)_c 
= \prod_d \Hom_R ( R\otimes {}_c\cC_d , M_d) 
= \prod_d \Hom ( {}_c\cC_d , M_d) \,,
\]
and their functoriality in $c$ is the evident one. Note also that the displayed products and coproducts are just finite direct sums, by the local boundedness of~$\cC$.
We therefore have a composite isomorphism
\[
(i^*M)_c 
= \bigoplus_d M_d \otimes {}_d\cC_c
\stackrel{\sim}{\to}
 \bigoplus_d M_d \otimes ({}_{ Sc}\cC_d)^*
\stackrel{\sim}{\to}
  \bigoplus_d \Hom ({}_{ Sc}\cC_d, M_d)
= (i^!M)_{ Sc}
\]
natural in $c$ and~$d$,
where the first isomorphism is induced by $\sigma\colon {}_d\cC_c\stackrel{\sim}{\to} ({}_{ Sc}\cC_d)^*$ and the second one by the fact that the $\base$-module ${}_{ Sc}\cC_d$ is finitely generated projective.
If we let $ S_*\colon \Mod \cC\stackrel{\sim}{\to} \Mod \cC$ denote the self-equivalence given by precomposition with the Serre functor, we have thus obtained an isomorphism $i^*\cong  S_* \circ i^!$ of functors $\cB\to \cA$.
In particular the faithful functor $i_*$ has the SGT-property
and therefore $\cC$ is $n'$-Gorenstein for some $n'\leq n$ by Proposition \ref{prop:Gor_Gor_general}.

Now assume that  for each object $c$ of~$\cC$ the unit $1_R\otimes 1_c \colon R \to R\otimes {}_c\cC_c$ has a retraction $\rho_c\colon R\otimes {}_c\cC_c\to R$ commuting with the right $R$-actions.
Notice that the unit $\eta_M\colon M\to i_*i^*M$ of the adjunction $(i^*,i_*)$ is given  at each~$c$ by the $R$-homo\-morph\-ism $\eta_{c,M}$ defined by
 $\eta_{M,c}(m)= m\otimes 1_c$  (for $m\in M_c$, $M=(M_d)_{d\in \obj \cC}$). 
 Hence the formula
\[
M_d\otimes {}_d\cC_c \;\ni\;
m \otimes f \longmapsto \begin{cases}
m \rho_c(f) & \textrm{if } c=d \\
0  & \textrm{otherwise,}
\end{cases}
\]
provides a natural splitting $\rho \colon i_*i^*\to \id$ of the unit~$\eta$, and  by Proposition~\ref{prop:equal_Gd} we conclude that in this case $n'=n$.
\end{proof}

\begin{remark}
Note that a $\base$-linear retraction $\rho_c\colon {}_c\cC_c \to \base$ of the unit $1_c\colon \base\to {}_c\cC_c$ extends uniquely to an $R$-linear retraction $\rho_c\colon R\otimes {}_c\cC_c\to R$ for the unit $1=1_R\otimes 1_c\colon R\to R\otimes {}_c\cC_c$, as in the theorem, by the formula $\rho (r\otimes f):= \rho(f)r$ ($r\in R, f\in {}_c\cC_c$). Thus in order to  have the equality $n'=n$ it suffices that each unit map of $\cC$ splits $\base$-linearly. This is always satisfied, for instance, if $\base$ is a field.
\end{remark}

It will also be useful in \S\ref{sec:KK} to recognize Gorenstein projective modules over $R \otimes \cC$.

\begin{cor} \label{cor:Serre_GProj}
In the situation of Theorem~\ref{thm:Serre_Gor}, an $R\otimes\cC$-module $M$ is Gorenstein projective (resp.\ Gorenstein injective) if and only if the evaluation $\ev_d(M)$ is a Gorenstein projective (resp.\ Gorenstein injective) $R$-module for each $d \in \cC$.
\end{cor}

\begin{proof}
This is just the second part of Proposition~\ref{prop:i_*GProj} since the functor $i_*$ in the proof of the theorem has been shown to have the SGT-property. 
\end{proof}

\begin{example}[Triangulated categories with Serre functor] \label{ex:tria-Serre}
Let $\base$ be an $n$-Goren\-stein commutative ring, and let $\cT$ be a Hom-finite $\base$-linear triangulated category equipped with a Serre functor.
Any small full subcategory $\cC$ of $\cT$ which is locally bounded, (essentially) closed under the Serre functor, and such that ${}_c\cC_c= \base\cdot 1_c$ for all $c\in \cC$, satisfies all the hypotheses of the theorem and is therefore $n$-Gorenstein.
\end{example}

Here is an amusing point of view on the usual example of group algebras.

\begin{example}[Groupoid categories] \label{ex:groupalgebras}
If $\cG$ is a finite groupoid and $\base$ an $n$-Gorenstein commutative ring, then the groupoid category $\base \cG$ (i.e.\ the free $\base$-category on~$\cG$) is also $n$-Gorenstein. (Thus for example, the integral group algebra $\Z G$ of a finite group $G$ is $1$-Gorenstein. This result goes back to~\cite{eilenberg-nakayama}; see also \cite{enochs-jenda:gorenstein-balance}*{Remark below Lemma 1.6}.) 
Indeed this follows from the theorem. First of all note the identity map $1_c\colon \base \to \base {}_c\cG_c$ of each object has the $\base$-linear retraction $\ev_{1_c}$ that picks the coefficient of the identity and plays the role of $\lambda_c$ from Remark~\ref{rem:trace-Serre}. Secondly, there is a Serre functor given by $S=\id$ and the natural isomorphism $\sigma\colon \base {}_d\cG_c \stackrel{\sim}{\to} (\base {}_c\cG_d)^*$ where $\sigma (g)(h):= \ev_{1_c}(hg)= \ev_{1_d}(gh)= \delta_{g^{-1},h}$ (for $g \in {}_d\cG_c$ and $h\in {}_c\cG_d$).
\end{example}

\begin{example}[Periodic complexes] \label{ex:complexes}
Let $\pi\geq0$ be a nonnegative integer. Let $\cC$ be the $\base$-category generated by the following quiver with vertex set $\{c_i \mid i\in \Z/\pi\Z \}$
\[
\xymatrix{
\cdots & \ar[l]_-{d_{i-1}} c_{i-1} & \ar[l]_-{d_i} c_{i} & \ar[l]_-{d_{i+1}} \cdots
}
\]
and with the relations $d_{i-1}d_i=0$ for all~$i$.
We can, and will, identify $\Mod \cC$ with the category $\Ch_\pi(\base)$ of $\pi$-periodic complexes over the base ring~$\base$, simply by rewriting a representation $M$ of $\cC$ as the cohomological complex $(M^i,d^i)$ with $M^i=M(c_i)$ and $d^i:= M(d_{i+1})$.
More generally if $R$ is any (not necessarily commutative) $\base$-algebra, we can identify $\Mod (R\otimes \cC)$ with $\Ch_\pi(R)$.

For $\pi=0$, this specializes to the usual category of complexes of right $R$-modules.  For $\pi=1$, we obtain the category of differential $R$-modules: pairs $(M,d)$ where $M$ is an $R$-module and $d\colon M\to M$ an $R$-linear map satisfying $d^2=0$.

\begin{prop}
\label{prop:Gorenstein_cplx}
If $R$ is an $n$-Gorenstein ring then the abelian category $\Ch_\pi(R)$ of $\pi$-periodic complexes is $n$-Gorenstein, for any integer $\pi\geq0$.
\end{prop}

\begin{proof}
Note that degree shifting, namely $Sc_i:= c_{i-1}$ and $Sd_i:= d_{i-1}$, defines an automorphism $S\colon \cC\stackrel{\sim}{\to} \cC$.
Since ${}_{c_{i-1}}\cC_{c_i} = d_i \base$, we choose $\lambda_{c_i}\colon {}_{c_{i-1}}\cC_{c_i} \to \base$ as in Remark~\ref{rem:trace-Serre} to act by sending $d_i \mapsto 1$. For all $i,j$ we then obtain an isomorphism
\[
\sigma_{c_i,c_j}\colon {}_{c_j}\cC_{c_i}\stackrel{\sim}{\longrightarrow} ({}_{Sc_i}\cC_{c_j})^* 
= ({}_{c_{i-1}}\cC_{c_j})^*
\]
by $\sigma(1_{c_i}):= (d_{i}\mapsto 1)$ (for $j=i$) and $\sigma(d_{i}):=(1_{c_{i-1}} \mapsto 1)$ (for $j=i-1$).
Thus we have just defined a Serre functor $(S,\sigma)$ on~$\cC$. 
Since $\cC$ is evidently Hom finite and locally bounded with split units $\base \to {}_c \cC_c$ we conclude with Theorem~\ref{thm:Serre_Gor}.
\end{proof}
For $\pi=0$ and $R$ Iwanaga-Gorenstein this example can already be found in~\cite{enochs-garciarozas-gorenstein_cplx}.
\end{example}

A more involved example arising from the filtrated K-theory of C*-algebra fibered over certain spaces~\cite{meyer-nest:filtrated2} will be described in Section~\ref{subsec:filtratedKK}.

%%%
\section{The boundary and the GT-property}

As usual, let $\mathcal C$ denote a small $\base$-category.
In the previous section we have specialized the GT-property to a criterion for importing Gorensteinness to $\cC$ from the base ring. We now consider an alternative incarnation of the GT-property.

\begin{defi}
Define the \emph{boundary} $\partial \mathcal C:= \coprod_{c\in \cC}{}_c \cC_c$ to be the category consisting of the coproduct\footnote{What we need here is the coproduct of $\base$-categories, which must again be a $\base$-category. Concretely, it is given by the disjoint union with an additional zero map between any two objects belonging to different components.} of all the endomorphism rings of~$\cC$, and write $i\colon \partial \cC\to \cC$ for the inclusion functor. We have the following adjoint functors
\begin{equation*}
\xymatrix{
\Mod \cC
 \ar[d]|{i_*} \\
\Mod \partial\cC
 \ar@/^10pt/[u]^{i^*}
 \ar@/_10pt/[u]_{i^!}
}
\end{equation*}
where $i_*$ denotes the restriction along~$i$ and where $i^*$ and $i^!$ are its left, respectively right, adjoint functor obtained by taking left and right Kan extensions.
 We refer to \cite{maclane}*{X.4} for Kan extensions as coends and \cite{maclane}*{IX.10} for the coend integral notation, which is used below.
\end{defi}

\begin{remark}
\label{rem:FD}
Note that $\partial \cC$ is Gorenstein if and only if each endomorphism ring ${}_c\cC_c$ is Gorenstein and their injective dimension is bounded with respect to~$c\in \cC$.
Note also that, if $\partial \cC$ is Gorenstein, the identification $\Mod \partial \cC= \prod_{c\in \cC}\Mod {}_c\cC_c$ restricts to $\FD \partial \cC = \prod_{c\in \cC}\FD {}_c\cC_c$. 
This is because projective objects and exact sequences are determined componentwise in the product category, and because of the boundedness condition with respect to the components.
\end{remark}

We are  interested in finding structural conditions on~$\cC$ which would guarantee that $i_*\colon \Mod \cC\to \Mod \partial \cC$ has the Gorenstein transfer property (Def.\,\ref{defi:GTP}), and which are sufficiently concrete to be verifiable in practice. 
Here, as in the previous section, we will in fact prove the \emph{strong} GT-property.

The next lemma shows that in this situation the Gorenstein dimension cannot drop. 

\begin{lemma}\label{lemma:equal_Gor_dim}
The unit $\eta_M\colon M\to i_*i^*M$ is a split monomorphism for all $M\in \Mod \partial \cC$.
\end{lemma}

\begin{proof}
Let us check the assertion directly, by computing the unit $\eta_M$ for an arbitrary representation $M$ of~$\partial \cC$.
Since in $\partial \cC$ there are no nonzero maps between distinct objects,
the value of $i_*i^*M$ at $c\in \partial \cC$ is simply the coproduct
\begin{align*}
(i^*M)_c 
&= \int^{d\in \partial \cC} M_d\otimes {}_d\cC_c \\
&= \coprod_{d\in \partial \cC} M_d \otimes_{{}_d\cC_d} {}_d \cC_c
\end{align*}
and $\partial\cC$ acts on it via ${}_c\cC_c$ from the right.
The unit $\eta_M=\eta_{M,c}\colon M_c\to (i_*i^*M)_c$, $m\mapsto m\otimes 1_c$,  is just the canonical inclusion 
of the component $M_c\cong M_c\otimes_{{}_c\cC_c} {}_c\cC_c$ at~$c$, which of course is a split monomorphism. 
\end{proof}

Now recall Definition~\ref{defi:neighbourhood}. Here is this section's criterion:

\begin{thm}
\label{thm:Gorenstein}
Let $\cC$ be a small $\base$-category satisfying the following hypotheses:
\begin{enumerate}
\item $\partial \cC$ is Gorenstein. 
\item $\cC$ is locally bounded.
\item There is an automorphism $S\colon \cC\stackrel{\sim}{\to} \cC$ restricting to an isomorphism
\[
S \colon 
\incoming(c)%=\{d\in \cC\mid {}_d\cC_c\neq 0\} 
\stackrel{\sim}{\longrightarrow}
\outcoming(c)%=\{d\in \cC\mid {}_c\cC_d\neq 0\}
\]
for every object $c\in\cC$. 
\item For all $c,d\in \cC$, there is an isomorphism of left ${}_d\cC_d$-modules
\[  
\psi_{c,d} \colon {}_d\cC_c
 \stackrel{\sim}{\longrightarrow} 
 \Hom_{{}_{d}\cC_{d}}({}_{Sc}\cC_{d}, {}_{d}\cC_{d}) =: ({}_{Sc}\cC_d)^\vee
\]
natural in $c\in \cC$. 
The latter means that $\psi_d\colon {}_d\cC \stackrel{\sim}{\to} ({}_{S(-)}\cC_d)^\vee$ as
${}_d\cC_d$-$\cC$-bimodules. 
\item For all $c,d\in \cC$, the Hom bimodule ${}_d\cC_c$ is a finitely generated Gorenstein projective right ${}_c\cC_c$-module. 
\end{enumerate}
Then $\cC$ is Gorenstein, of the same Gorenstein dimension as~$\partial C$.
\end{thm}

\begin{remark} \label{rem:trace-boundary}
In the same vein as Remark~\ref{rem:trace-Serre}, in this case we can also replace $(\psi_{c,d})_{c,d\in \cC}$ by a family
\[
\mu_d\colon {}_{Sd}\cC_d \longrightarrow {}_d\cC_d
\]
indexed by the objects of $\cC$. We again put $\mu_d = \psi_{d,d}(1_d)$ and can reconstruct $\psi_{c,d}(f)(g) = \mu_d(S(f)g)$. All the required properties of the morphisms $\psi_{c,d}$, except the fact that they are isomorphisms, are then equivalent to requiring that each $\mu_d$ is a ${}_d\cC_d$-bimodule homomorphism.
\end{remark}

\begin{proof}[Proof of Theorem~\ref{thm:Gorenstein}]
The functor $i_*$ is faithful, since $\partial \cC$ contains all the objects of~$\cC$. 
We are going to show that $i^*|_{\FD \partial\cC}\cong (S_* \circ i^!)|_{\FD \partial\cC}$, so that in particular the functor $i_*$ has the GT-property. 
The conclusion of the theorem will then follow by combining hypothesis~(1), Lemma \ref{lemma:equal_Gor_dim} and Theorem~\ref{thm:GorGor_general}.

To prove that the two functors are isomorphic we compute them explicitly.
For every $M\in \Mod \partial \cC$ and every $c\in \cC$, we have:
\begin{align*}
(i^*M)_c & = \int^{d\in \partial\cC} M_d \otimes {}_d\cC_c & \\
& = \coprod_{d\in \partial\cC} M_d \otimes_{{}_d\cC_d} {}_d\cC_c
  & \textrm{because } \partial\cC(a,b)=0 \textrm{ for } a\neq b \\
& =  \bigoplus_{d\in \outcoming(c)} M_d \otimes_{{}_d\cC_d} {}_d\cC_c  
  & \textrm{ by hypothesis (2)\,.}  
\end{align*}
Dually, we obtain:
\begin{align*}
(S_*i^!M)_c 
 & = (i^!M)_{Sc} & \\
& = \int_{d\in \partial\cC} \Hom( {}_{Sc}\cC_{d}, M_{d}) & \\
& = \prod_{d\in \partial\cC} \Hom_{{}_{d}\cC_{d}}( {}_{Sc}\cC_{d}, M_{d})
  & \textrm{because } \partial\cC(a,b)=0 \textrm{ for } a\neq b \\
& =  \bigoplus_{d\in \incoming(Sc)} \Hom_{{}_{d}\cC_{d}}( {}_{Sc}\cC_{d}, M_{d})
  & \textrm{ by hypothesis (2) \,.}
\end{align*}
Now, consider the canonical evaluation morphism
\[
\epsilon_{c,d}\colon M_d\otimes_{{}_d\cC_d}({}_{Sc}\cC_d)^\vee
\longrightarrow
 \Hom_{{}_{d}\cC_{d}}({}_{Sc}\cC_{d}, M_{d})
\]
sending $m\otimes f$ to $(g\mapsto m\cdot f(g))$, which is natural in all the evident ways. 
In particular, it is natural in $M_d\in \Mod {}_d\cC_d$.
Note that for $M_d={}_d\cC_d$ it is an isomorphism; since ${}_{Sc}\cC_{d}$ is finitely generated, by~(5), this extends to all $M_d\in \Proj {}_d\cC_d$.
Moreover, we claim that both $(-) \otimes_{{}_d\cC_d} ({}_{Sc}\cC_d)^\vee$ and $\Hom_{{}_{d}\cC_{d}}({}_{Sc}\cC_{d}, -)$ send short exact sequences in $\FD {}_d\cC_d$ to right exact sequences.
For the first functor that's clear, and for $\Hom_{{}_{d}\cC_{d}}({}_{Sc}\cC_{d}, -)$ this holds because, by hypothesis~(5), the right ${}_d\cC_d$-module ${}_{Sc}\cC_d$ is Gorenstein projective. 
Since $\Proj {}_d\cC_d$ generates $\FD {}_d\cC_d$ by taking cokernels of monomorphisms, we conclude from the right exactness of the functors and the invertibility of $\epsilon_{c,d}$ on $\Proj {}_d\cC_d$ that $\epsilon_{c,d}$ is invertible for all $M_d\in \FD {}_d\cC_d$.

By composing this isomorphism with the the one in~(4), we obtain the following:
\[
\xymatrix{
\varphi_{c,d}\colon 
M_d \otimes_{{}_d\cC_d} {}_d\cC_c
\ar[r]^-{\id\otimes \psi_{c,d}}_-\sim &
 M_d\otimes_{{}_d\cC_d}({}_{Sc}\cC_d)^\vee
 \ar[r]^-{\epsilon_{c,d}}_-\sim &  \Hom_{{}_{d}\cC_{d}}({}_{Sc}\cC_{d}, M_{d}) \,.
 }
\]
These assemble into isomorphisms
\begin{equation*}
\xymatrix{
\bigoplus_{d} M_d \otimes_{ {}_d\cC_d } {}_d\cC_c \ar[r]^-{\varphi_{c}} &
\bigoplus_d \Hom_{{}_d\cC_d}({}_{Sc}\cC_d, M_d)
}
\end{equation*}
for $M=(M_d)_d\in \FD \partial\cC = \prod_d \FD {}_d\cC_d$.
From the naturality in~$c\in \cC$, hypothesis~(3), and the above computations we conclude that these maps define an isomorphism $\varphi \colon i^*M \stackrel{\sim}{\to} S_* i^! M$, natural in $M\in \FD\partial\cC$.
This completes the proof.
\end{proof}

Analogously to \S\ref{sec:Serre}, it is useful to be able to spot Gorenstein projective $\cC$-modules easily (see, however, Remark~\ref{rem:eq-KK-Gore}).

\begin{cor} \label{cor:boundary_GProj}
In the set up of Theorem~\ref{thm:Gorenstein}, a $\cC$-module $M$ is Gorenstein projective (resp.\ Gorenstein injective) if and only if $M_d \in \GProj {_d\cC_d}$ (resp.\ $M_d \in \GInj {_d\cC_d}$) for each $d \in \cC$.
\end{cor}

\begin{proof}
This is simply the second part of Proposition~\ref{prop:i_*GProj} since the functor $i_*$ in the proof of the theorem has been shown to have the SGT-property. 
\end{proof}

\begin{remark}
Our main motivation for this criterion was a rather non-trivial example of K\"ohler~\cite{koehler} arising from the equivariant KK-theory of C*-algebras, which we defer to Section~\ref{subsec:equivariantKK}.
\end{remark}

%%%
\section{Triangulated categories with few morphisms are Gorenstein}
\label{sec:triangulated_few}

We now show that many essentially small triangulated categories are Gorenstein. 
The important case $n=0$ of this result was first proved by Neeman in \cite{neeman:adams_brown}*{Theorem~5.1}.
For compact objects in a compactly generated triangulated category and arbitrary $n\geq0$ it was observed by Beligiannis, see \cite{beligiannis:relative}*{Cor.\,11.3}. (Neither author couched the result in terms of Gorenstein algebra as we do.)

Here it suffices to work with $\base=\Z$.

\begin{thm} \label{thm:triangulated_n_Gorenstein}
Let $\cC$ be a triangulated category admitting a skeleton with at most $\aleph_n$ morphisms,
where $n\in \N$ is any nonnegative integer. 
Then $\Mod \cC$ is a Gorenstein category of Gorenstein dimension at most~$n+1$, and the $\cC$-modules of finite projective dimension are precisely the flat ones, i.e., the homological functors.
\end{thm}

Let us record a consequence for phantom maps.
Recall that a map $X\to Y$ in a compactly generated triangulated category~$\cT$ is \emph{phantom} if all composites $C\to X\to Y$ are zero whenever $C$ is compact. 
In other words, the ideal of all phantom maps in~$\cT$ is precisely the kernel of the restricted Yoneda functor  $\cT \longrightarrow \Mod \cT^c$, 
 $X\mapsto \cT(-,X)|_{\cT^c}$.

\begin{cor} \label{cor:phantom_vanishing}
Let $\cT$ be a compactly generated triangulated category such that its subcategory $\cT^c$ of compact objects has at most $\aleph_n$ isomorphism classes of morphisms.
Then any composite of $n+2$ consecutive phantom maps  of~$\cT$ is zero.
\end{cor}

\begin{proof}
Let $X\in \cT$. By applying the theorem to $\cC= \cT^c$, it follows that the homological functor  $h_\cC(X):= \cT(-,X)|_{\cC}$ on $\cC$ has projective dimension at most $n+1$ in $\Mod \cC$.
The rest is a standard argument, which goes as follows.
By lifting to $\cT$ a projective resolution of~$h_\cC(X)$ of length~$n+1$, we see that $X$ has phantom length $\leq n+2$, i.e.,  that $X$ can be recursively constructed from compact objects in $n+2$ steps by taking coproducts, extensions and direct summands
(see \cite{christensen:ideals}*{Prop.\,4.7}). It follows by the Ghost Lemma that any composite of $n+2$ phantoms starting at $X$ must be zero (see~\cite{beligiannis:ghost}). Since $X$ was arbitrary, this proves the claim.
\end{proof}

Before proving the theorem we recall some well-known facts.

\begin{lemma} \label{lemma:equiv_modules_triangulated}
If $\cC$ is an essentially small triangulated category, the following conditions are equivalent for a module $M\in \Mod \cC$:
\begin{enumerate}
\item $M$ is flat.
\item $M$ is a homological functor.
\item $M$ is fp-injective (i.e.\ $\Ext^i_\cC(F,M)=0$ for all $i\geq 1$ and all finitely presented~$F$).
\item $M$ is a filtered colimit of representable functors.
\end{enumerate}
It follows in particular that the subcategory of flat modules contains all injectives and projectives and is closed under the formation of cokernels of monomorphisms.
\end{lemma}

\begin{proof}
The equivalences (1)$\Leftrightarrow$(2)$\Leftrightarrow$(3) can be found e.g.\ in \cite{krause:telescope}*{Lemma~2.7}, and rest on the equivalence (1)$\Leftrightarrow$(4) which is proved in \cite{oberst-rohrl:flat_coherent}*{Theorem~3.2}. (See also \cite{neeman:brown_phantomless}*{Lemma~1} for a direct proof of (2)$\Leftrightarrow$(4).)
Clearly, all injectives are \emph{a fortiori} fp-injective, and all projectives are flat. 
Moreover, it follows immediately from the Ext long exact sequence that the class of fp-injectives is closed under cokernels of monomorphisms.
\end{proof}

\begin{proof}[Proof of Theorem~\ref{thm:triangulated_n_Gorenstein}]
The main point is a consequence of \cite{simson:pgldim}*{Corollary 3.13}, which says that if the category~$\cC$ has a skeleton with at most $\aleph_n$ arrows then every flat $\cC$-module has $\cC$-projective dimension at most~$n+1$ (this fact does not use yet that $\cC$ is triangulated).
Since $\cC$ is triangulated, we can use Lemma~\ref{lemma:equiv_modules_triangulated}. 
We deduce easily that if a $\cC$-module admits a finite projective resolution than it must be flat. Hence in this case the flat modules are precisely the modules of finite projective dimension.
Also, by Lemma~\ref{lemma:equiv_modules_triangulated}, every injective $\cC$-module is flat and therefore has projective dimension at most $n+1$. 
Dually, let $P$ be a projective $\cC$-module. 
Choose an injective resolution $P \smash{\stackrel{\sim}{\to}} I^\bullet$  and consider its good truncation
$0\to P \to I^0 \to I^1 \to \ldots \to I^n \to I^{n+1}\to Q\to 0$. 
This exact sequence represents an element of 
$\Ext^{n+2}_\cC(Q,P)$.
The projective $P$ is flat, and by Lemma~\ref{lemma:equiv_modules_triangulated} so are all the $I^n$ and all cosyzygies of~$P$, including~$Q$. 
Hence $Q$ has projective dimension at most $n+1$; 
hence $\Ext^{n+2}_\cC(Q,P)=0$; hence the truncated resolution splits; hence~$P$ has an injective resolution of length~$n+1$. 

Thus every injective has projective dimension at most~$n+1$ and every projective has injective dimension at most~$n+1$, and we conclude with Lemma~\ref{lemma:Gore_criterion}.
\end{proof}

\begin{remark} \label{rem:pure_semisimple}
The upper bound $n+1$ in Theorem~\ref{thm:triangulated_n_Gorenstein} is not always reached. 
Consider for instance the case $n=0$, i.e.\ when $\cC$ has a countable skeleton. 
Assume moreover that $\cC=\cT^c$ is the category of compact objects in a compactly generated triangulated category.
Then the Gorenstein dimension of $\cC$ is at most one, by the theorem. It is equal to zero (even without the need for any set-theoretical conditions) \iff $\cT$ is \emph{pure semi-simple}, i.e.: \iff $\cT$ has no non-zero phantoms, \iff the restricted Yoneda functor $h_\cC$ is fully-faithful, \iff every object of $\cT$ is a retract of a direct sum of compacts, \iff $h_\cC X$ is injective for all $X\in \cT$ (see \cite{krause:telescope}*{\S2.4} and \cite{beligiannis:relative}*{\S9} for these and other characterizations of pure semi-simplicity).
Typical examples of pure semi-simple triangulated categories are the categories $\cT = \underline{\GProj}R \simeq \Mod R/\Proj R$, where $R$ is a commutative artinian principal ideal ring ring such as $\Z/(p^n)$ for a prime $p$ or $k[x]/(x^n)$ for a field $k$, and also $\cT = \D(kQ)$ where $k$ is a field and $Q$ a is Dynkin quiver. We refer to~\cite{krause:loc-fin-tria} for more details and examples.
\end{remark}

\begin{example} \label{ex:spectra}
Let $\cT$ be the homotopy category of spectra. It is well-known that the subcategory of finite spectra $\cC=\cT^c$ admits a countable skeleton and that $\cT$ is not pure semi-simple (see~\cite{margolis:spectra}). Hence $\Mod \cT^c$ is a 1-Gorenstein category.
\end{example}

\begin{example} \label{ex:countable_rings}
If $\cT=\D(R)$ is the derived category of a countable ring~$R$, then $\cT^c$ has a countable skeleton and is therefore either pure semisimple or $1$-Gorenstein. To see this, use that $\cT^c$ is the thick subcategory of $\D(R)$ generated by~$R$ and proceed by induction on the length of objects $X,Y\in \cT^c$ to show that $\cT(X,Y)$ is countable. The same induction also shows that there are only countably many isomorphism classes of objects in~$\cT^c$, hence $\cT^c$ has a countable skeleton.
This example works just as well when considering periodic complexes, as in Example~\ref{ex:complexes}.
\end{example}

\begin{example} \label{ex:pure_periodic}
Let $\cT= \D_\pi(\Z)$ denote the derived category of $\pi$-periodic complexes of abelian groups (as before, we allow the case $\pi=0$ of ordinary, i.e.\ not necessarily periodic, complexes). We know from Example~\ref{ex:countable_rings} that $\cT^c$ is $\leq 1$-Gorenstein. 
It is in fact 1-Gorenstein, since not every object is a retract of a direct sum of compacts, e.g.~$\Q$.
Consider now the full subcategory $\cC:=\{\Sigma^i \Z/n \mid i\in \Z/\pi , n\geq 0\} \subset \cT^c$.
Since $\Z$ is hereditary, every complex of abelian groups is quasi-isomorphic to its cohomology. It follows easily from this that every indecomposable object of $\cT^c$ is isomorphic to one in~$\cC$. Hence $\cT^c= \add \cC$, and the inclusion $\cC\hookrightarrow \cT^c$ induces an equivalence $\Mod \cT^c \cong \Mod \cC$.
In view of Section~\ref{subsec:umct}, we now explicitly describe the category $\cC$ in the case~$\pi=2$\,:
\begin{align*}
\Hom(\Z,\Z)&\cong \Z 
 &\quad  \Hom(\Z,\Sigma\Z)&\cong 0\\
\Hom(\Z,\Z/a)&\cong \Z/a 
 &\quad \Hom(\Z,\Sigma\Z/a)&\cong 0 \\
\Hom(\Z/a,\Z)&\cong 0 
 &\quad \Hom(\Z/a,\Sigma\Z)&\cong \Z/a \\
\Hom(\Z/a,\Z/b)&\cong \Z/ (a,b) 
 &\quad \Hom(\Z/a,\Sigma\Z/b)&\cong \Z/ (a,b)
\end{align*}
Here $(a,b)$ denotes the greatest common divisor of $a,b\geq1$. Composition in $\cC$ is computed by multiplication in the evident way. 
\end{example}

%%%
\section{Exactness and finite projective dimension}
\label{sec:exact_vs_fd}

Let $\cC$ be a small $\base$-category.
In this section, we assume that $\Mod \cC$ is $n$-Gorenstein and locally coherent.
Our goal is to find a criterion for a $\cC$-module to be in $\FD \cC$, i.e.\ to have finite projective dimension.

\begin{defi} \label{defi:exact_seq}
Let 
\[
F\colon \quad \quad
\xymatrix@1{
\cdots \ar[r]^-{d_{i+1}} &
 F_{i+1} \ar[r]^-{d_i} &
  F_i \ar[r]^-{d_{i-1}} &
   \cdots
}
\]
be an acyclic complex of finitely generated projective right $\cC$-modules i.e., $F_i=\Img(e_i)$ where $e_i$ is an idempotent endomorphism of some $ \oplus_{j\in J_i} ({}_{c_{i,j}}\cC)$ with $J_i$ finite. 
We call $F$ an \emph{exact sequence in $\cC$} (this is a slight abuse of language, but will be justified in many applications).
We say that a right $\cC$-module $M$ is \emph{$F$-exact} if the complex $\Hom(F,M)$ is acyclic. 
By Yoneda this is simply saying the complex
\[
\xymatrix{
M(F)\colon &  \cdots \ar[r]^-{M(d_{i-1})} & M(F_i) \ar[r]^-{M(d_i)} & M(F_{i+1}) \ar[r]^-{M(d_{i+1})} & \cdots
}
\]
is exact in $\Mod \base$ (recall our conventions \S\ref{subsec:conventions}). 
\end{defi}

\begin{lemma}
\label{lemma:exact_necessary}
Every $M\in \FD \cC$ is $F$-exact for all exact sequences $F$ in~$\cC$.
\end{lemma}

\begin{proof}
Given $F$ an exact sequence in~$\cC$ and $M\in \FD\cC$ as in the statement it is standard that $M(F) = \Hom(F,M)$ is exact. Indeed, brutally truncating $F$ at any point gives a projective resolution of the corresponding syzygy and so the cohomology of $M(F)$ is just computing Ext groups. Since $M$ has finite projective and hence injective dimension and we can truncate arbitrarily far along~$F$, the complex $M(F)$ must therefore be acyclic.
\end{proof}

We now wish to prove a sort of converse to this assertion for sufficiently nice~$\cC$, i.e.\ we will give a criterion for $\cC$ to have `enough exact sequences' to detect representations of finite projective dimension in terms of exactness relative to these sequences. 

\begin{prop} \label{prop:enough_exact}
Let $\cS$ be a subset of $\modu \cC$. Then the following are equivalent:
\begin{enumerate}
\item $\FD \cC = \ker \Ext^1_\cC(\cS , -)$.
\item $\Gproj \cC$ is the closure of $\cS\cup \proj \cC$ under taking extensions and retracts.
\end{enumerate}
\end{prop}

\begin{proof}
To see why (2) implies (1), note that $\ker \Ext^1_\cC(\cS , -)$ does not change if we add to $\cS$ extensions and retracts of modules in~$\cS$. Hence $\ker \Ext^1_\cC(\cS , -) = \ker \Ext^1_\cC(\Gproj \cC , -)$ and the latter implies $\FD \cC \supseteq \ker \Ext^1_\cC(\cS , -)$ by Theorem~\ref{thm:triangulated_small}. 
The reverse inclusion is always true by using the appropriate cotorsion pair from Proposition~\ref{prop:pairs}.
Now assume~(1). By the cotorsion pair again, $\cS\subseteq \Gproj \cC$. By Theorem \ref{thm:triangulated_small}~(iii), $\cS$ is a set of compact objects in the triangulated category $\underline{\GProj}\,\cC$ such that its right Hom-orthogonal is zero because of~(1).
By \cite{stovicek:classification}*{Theorem 3.7~(ii)}, all compacts, that is the objects of $\underline{\Gproj}\,\cC$, are obtained by extensions and retracts. Hence all objects of $\Gproj \cC$ are obtained from $\cS \cup \proj \cC$ by the corresponding operations performed in the abelian category $\modu \cC$.
\end{proof}

\begin{defi} \label{defi:enough}
A family $\cF$ of exact sequences in $\cC$ (Def.\,\ref{defi:exact_seq}) is said to be \emph{sufficient} if the following set of $\cC$-modules 
\[
\{ Z^i F \mid F \in \cF, i\in \Z\} \cup \proj \cC
\]
generates $ \Gproj \cC$ by taking extensions and retracts. 
If such a family exists, we also say that $\cC$ \emph{has enough exact sequences}. 
\end{defi}

\begin{remark} \label{rem:sufficient-sing}
Put another way, $\cF$ is a sufficient family if the syzygies of its sequences generate the singularity category $\D_{\mathrm{Sg}}(\cC)$ classically, i.e.\ by taking cones, retracts and (de)suspensions -- actually (de)\-sus\-pen\-sions are already included here, since we automatically have all syzygies and cosyzygies. This follows by inspecting the proof of Proposition~\ref{prop:enough_exact}.
\end{remark}

\begin{thm} \label{thm:enough_exact}
Let $\cC$ be a small category such that $\Mod \cC$ is Gorenstein and locally coherent and assume that $\cF$ is a sufficient family of exact sequences for~$\cC$. 
Then a $\cC$-module has finite projective dimension if and only if it is $F$-exact for all $F\in \cF$.
\end{thm}
\begin{proof}
If $M\in \FD \cC$ then $M$ is $F$-exact for all exact sequences $F$ by Lemma~\ref{lemma:exact_necessary}.
Conversely, let $M$ be $F$-exact for all $F\in \cF$. As in the proof of Lemma \ref{lemma:exact_necessary}, this means that $\Ext^1_\cC(Z^iF, M)=0$ for all $F\in \cF$. In view of the implication (2)$\Rightarrow$(1) in Proposition \ref{prop:enough_exact} we conclude from the sufficiency of $\cF$ that $M\in \FD \cC$.
\end{proof}

%%%
\section{Gorenstein closed subcategories of triangulated categories}
\label{sec:gore_closed}

In this section, let $\cT$ be any ($\base$-linear) triangulated category.
Assume we are given a small full subcategory $\cC$ which is closed under (de)suspension (i.e., $\Sigma \cC=\cC$) and such that $\Mod \cC$ is Gorenstein and locally coherent.

We write $\add \cC$ for the additive closure of $\cC$ in~$\cT$ (i.e.\ we close under all finite sums and those summands that exist in~$\cT$). The Yoneda embedding $\cC\to \proj \cC$ extends canonically to a  fully faithful embedding $\add \cC\to \proj \cC$, that we will treat as an inclusion by a slight abuse of notation.

\begin{lemma} \label{lemma:Im_Gore}
If $x\to y\to z\to \Sigma x$ is an exact triangle in $\cT$ with vertices belonging to $\add \cC$, then
$\Img(x \to y)$, where this image is taken in $\Mod \cC$, is a finitely generated Gorenstein projective $\cC$-module.
\end{lemma}

\begin{proof}
This is clear. It would be true even without assuming $\cC$ to be Gorenstein, since the exact triangle gives rise in $\Mod \cC$ to a complete projective resolution of any of its syzygies, by Yoneda.
\end{proof}

We now wish to characterize those $\cC$ which have enough exact sequences, in the sense of the previous section, where the exact sequences arise from distinguished triangles of~$\cT$.

\begin{prop} \label{prop:equiv_Gore_closed}
Let $\cC$ be as above. The following are equivalent:
\begin{enumerate}
\item There exists a sufficient family of exact sequences in $\cC$, as in Definition~\ref{defi:enough}, consisting of exact triangles in $\add \cC\subseteq \cT$.
\item There is a subset $\mathcal S\subseteq \Gproj \cC$ such that $\cS\cup \proj \cC$ generates $\Gproj \cC$ by extensions and retracts, and consists of modules as in Lemma \ref{lemma:Im_Gore}; that is, for each $M\in \mathcal S$ there exists an exact triangle $x\to y\to z\to \Sigma x$ in $\add \cC$ such that $M\cong \Img(x\to y)$.
\item For each $M\in \Gproj \cC$ there exists an exact triangle $x\to y\to z\to \Sigma x$ in $\add \cC$ such that $M\cong \Img(x\to y)$.
\item If $x\to y$ is a morphism in $\add \cC$ such that $\Img(x\to y)\in \Gproj \cC$, then $\cone(x\to y)\in \add \cC$.
\end{enumerate}
\end{prop}

\begin{defi} \label{defi:Gore_closed}
We say that $\cC$ is a \emph{Gorenstein closed} subcategory of $\cT$ if the four equivalent conditions of the proposition hold.
\end{defi}

The proof of the proposition requires some preparation.

\begin{lemma} \label{lemma:cones}
Let $M\in \Gproj \cC$ and let $f_i\colon x_i\to y_i$ \textup($i=1,2$\textup) be two morphisms in $\add \cC$ whose images are isomorphic to~$M$. Then there exist $p_1, p_2\in \add \cC$ and an isomorphism $\cone (f_1)\oplus p_1\cong \cone (f_2)\oplus p_2$.
\end{lemma}
\begin{proof}
Since $M$ is Gorenstein projective, we may extend each $f_i$ to a complete projective resolution $P_i^\bullet$ of~$M$ having $f_i$ as differential in degree zero. 
Without loss of generality, we may assume that the new objects in these complexes are finitely generated free.
Let us denote by $\Ch_\mathrm{ac}(\add \cC)$ the full subcategory of $\Ch(\add \cC)$ consisting of the complexes which are acyclic in $\Mod\cC$. Recall that then the zero cocycles functor $Z^0\colon \Ch_\mathrm{ac}(\add \cC) \to \Gproj\cC$ between the Frobenius exact categories is well known to descend to a triangle equivalence $\Ho_\mathrm{ac}(\add \cC) \overset{\sim}\to \underline{\Gproj} \,\cC$ (this is in fact the Quillen equivalence from Remark~\ref{rem:sing}). Hence $P_1^\bullet \cong P_2^\bullet$ in $\Ho(\add \cC)$ and there exist in $\Ch (\add \cC)$ contractible complexes $C^\bullet_i$ ($i=1,2$) and an isomorphism $P^\bullet_1\oplus C^\bullet_1 \cong P^\bullet_2\oplus C^\bullet_2$.
Thus in particular we have an isomorphism of morphisms in $\add \cC$, as follows:
\[
\xymatrix{
x_1 \oplus C_1^0 \ar[r]^-{f_1 \oplus d_{C_1}^0} \ar[d]_{\sim} &
 y_1 \oplus C_1^1 \ar[d]^\sim \\
x_2 \oplus C_2^0 \ar[r]^-{f_2 \oplus d_{C_2}^0} & y_2 \oplus C_2^1
}
\]
By taking cones in $\cT$ we obtain an isomorphism 
$\cone(f_1\oplus d^0_{C_1}) \cong \cone(f_2\oplus d^0_{C_2})$.
Denoting $p_i:=\cone (d^0_{C_i})$, we deduce
$ \cone(f_1)\oplus p_1\cong \cone(f_2)\oplus p_2$.
It remains to show that $p_1,p_2\in \add\cC$.
This is an immediate consequence of Lemma~\ref{lemma:pseudo_inv} below (a general fact about triangulated categories) because the component $C^1_i\to C^0_i$ of any chain homotopy $\id_{C_i^\bullet}\sim 0$ is a pseudo-inverse for~$d^0_{C_i}$. 

(We now sketch another, less elementary proof of the latter fact. Embed $\cT$ into its idempotent completion~$\smash{\widetilde{\cT}}$; see~\cite{balmer-schlichting}. 
Let us use the notations $\add_\cT \cC$ and $\add_{\widetilde \cT} \cC$ to distinguish between the additive hulls of $\cC$ as taken in $\cT$ and in~$\smash{\widetilde{\cT}}$, respectively.
Since $\add_{\widetilde \cT}\cC$ is an idempotent complete additive category, every contractible complex over it splits. Thus $C^\bullet_i$ splits in $\Ch(\add_{\widetilde \cT}\cC)$, from which it follows that $p_i=\cone(d^0_{C_i}) \in \add_{\widetilde \cT}\cC$. But since $d^0_{C_i}$ was a morphism of $\cT$ we conclude that actually $p_i\in \cT \cap \add_{\widetilde \cT}\cC= \add_\cT \cC$, as required.)
\end{proof}

\begin{lemma} \label{lemma:pseudo_inv}
Assume that a morphism $f\colon X\to Y$ in a triangulated category admits a pseudo-inverse, i.e., a morphism $g\colon Y\to X$ such that $fgf=f$. Then $\cone(f)$ is a retract of $Y\oplus \Sigma X$.
\end{lemma}

\begin{proof}
Choose an exact triangle $(f,u,v)$ containing~$f$.
Since $(1_Y-fg)f=0$ and $\Sigma f(1_{\Sigma X}-\Sigma(gf))=0$,  we obtain factorizations $q$ and $r$ as follows:
\[
\xymatrix{
&&&\Sigma X \ar[d]|-{1- \Sigma(gf)} \ar@{..>}[dl]_r \ar[dr]^-0 & \\
X \ar[r]^-f \ar[dr]_0 &
 Y \ar[r]^-u  \ar[d]|-{1-fg} &
  Z \ar[r]^-v \ar@{..>}[dl]^{q} &
  \Sigma X \ar[r]_-{-\Sigma f} &
   \Sigma Y \\
& Y &&&
}
\]
In order to conclude it would now suffice to show that the composite
\[
\xymatrix{
Z \ar[r]^-{\left[ {}^q_v \right]} & Y \oplus \Sigma X \ar[r]^-{[u \, r]} & Z
} ,
\]
that is $uq+rv\colon Z\to Z$, is invertible.
But this holds because the diagram
\[
\xymatrix{
X \ar[d]_{1} \ar[r]^f & Y \ar[d]_{1} \ar[r]^-u & Z \ar[r]^-v \ar[d]_{uq+rv} & \Sigma X \ar[d]_{1} \\
X \ar[r]^f & Y \ar[r]^-u & Z \ar[r]^-v & \Sigma X
}
\]
is commutative, as a straightforward computation shows.
\end{proof}

\begin{proof}[Proof of Proposition \ref{prop:equiv_Gore_closed}]
The implication (1)$\Rightarrow$(2) is simply the definition of a sufficient family.
The implication (2)$\Rightarrow$(3) can be proved as follows.

Let $x_M\to y_M\to z_M\to \Sigma x_M$ be a triangle with vertices in $\add \cC$ and such that $\Img(x_M\to y_M)\cong M$, and assume that $M\cong M_1\oplus M_2$. Since $M_i\in \Gproj \cC$, we can find a complete projective resolution for  it, so in particular we can construct a triangle $x_i\to y_i\to z_i\to \Sigma x_i$ such that $x_i,y_i\in \add \cC$ and $\Img(x_i\to y_i)\cong M_i$ (for $i=1,2$). By summing the two we obtain the triangle 
$x_1\oplus x_2\to y_1\oplus y_2 \to z_1\oplus z_2 \to \Sigma(x_1\oplus x_2)$. 
Since $\Img(x_1\oplus x_2\to y_1\oplus y_2)\cong M$, it follows from Lemma~\ref{lemma:cones} that there exist $p,p'\in \add \cC$ and an isomorphism $z_M\oplus p \cong z_1\oplus z_2 \oplus p'$.
Hence we have $z_1,z_2\in \add \cC$. We conclude that retracts of $M$ are also syzygies of some triangles in $\add \cC$, provided $M$ is. It remains to prove the analogous fact for extensions.

Let $0\to K\to L\to M\to 0$ be an extension in $\Mod \cC$. Assume that  in $\add \cC$ there exist triangles $x_K\to y_K\to z_K\to \Sigma x_K$  and
$x_M\to y_M\to z_M\to \Sigma x_M$ such that  $\Img(x_K\to y_K)\cong K$ and $\Img(x_M\to y_M)\cong M$. We must show the existence in $\add \cC$ of a triangle
$x_L\to y_L\to z_L\to \Sigma x_L$ with $\Img(x_L\to y_L)\cong L$.
Reasoning as in the Horseshoe Lemma, we may construct the following commutative diagram, where the rows are split short exact sequences and the two outer columns are parts of the given triangles:
\[
\xymatrix{
x_K \ar[r] \ar[d]^d & x_K\oplus x_M \ar[r] \ar[d]^{\left[ {}^d_0 \, {}^f_d \right]} & x_M \ar[d]^d \\
y_K \ar[r] \ar[d]^d & y_K\oplus y_M \ar[r] \ar[d]^{\left[ {}^d_0 \, {}^g_d \right]} & y_M \ar[d]^d \\
z_K \ar[r]  & z_K \oplus z_M \ar[r] & z_M
}
\]
Since the middle vertical column composes to zero, we have the commutative square:
\[
\xymatrix{
x_M \ar[r]^-f \ar[d]_d & y_K \ar[d]^{-d} \\
y_M \ar[r]^-g & z_K
}
\]
By applying \cite{neeman:new}*{Theorem 1.8} (a form of the octahedron axiom) to this square, we can complete the above middle column to an exact triangle of the form $x_K\oplus x_M \to y_K\oplus y_M \to z_K\oplus z_M \to \Sigma(x_K\oplus x_M)$. Since  the image of the first map is~$L$, this proves the claim. 
This concludes the proof of (2)$\Rightarrow$(3).

The implication (3)$\Rightarrow$(4) is another easy application of Lemma \ref{lemma:cones}, just as above.

It remains only to show (4)$\Rightarrow$(1). 
Let $M\in \Gproj \cC$. Hence $M$ has a complete projective resolution, so in particular we can construct a triangle $x\to y\to z\to \Sigma x$ such that $x,y \in \add \cC$ and $\Img(x\to y)\cong M$. By~(4), we know that $z\in \add \cC$ as well. If we do this for a generating set of Gorenstein projectives, we will obtain a sufficient family of exact sequences in~$\cC$.
\end{proof}

\begin{thm} \label{thm:exact_detection}
Let $\cT$ be a triangulated category and let $\cC$ be an (essentially) small suspension closed full subcategory of $\cT$ such that $\Mod \cC$ is locally coherent and Gorenstein.
Then the following are equivalent: 
\begin{enumerate}
\item[(1)]  $\cC$ is Gorenstein closed in $\cT$ (Def.\,\ref{defi:Gore_closed});
\item[(2)] the restricted Yoneda functor
 $h_\cC \colon \cT \to \Mod \cC$, sending $X\in \cT $ to $\cT(-,X)|_\cC$, 
takes values in $\FD \cC$.
\end{enumerate}
\end{thm}

\begin{proof}
The implication (1)$\Rightarrow$(2) follows immediately by combining Theorem \ref{thm:enough_exact} with part (1) of Proposition \ref{prop:equiv_Gore_closed}. It remains to show the other implication.
Thus assume that (2) is true; we are going to verify condition (3) of Proposition \ref{prop:equiv_Gore_closed}, that is, that every finitely presented Gorenstein projective is a syzygy of some triangle in $\add \cC$.
 Let $M\in \Gproj \cC$. Then $M$ has a complete projective resolution, and in particular we can find in $\Mod \cC$ an exact sequence
\begin{align} \label{exact_seq_M}
\xymatrix{
 x_2 \ar[r]^-{d_2} &
  x_1 \ar[r]^-{d_1} &
   x_0 \ar[r] &
    M \ar[r] & 0
}
\end{align}
such that $x_i \in \add \cC$ for all~$i$ (again, we treat $h_\cC$ as an inclusion $\add \cC \subseteq \Mod \cC$).
Complete $d_1$ to a distinguished triangle in~$\cT$, as follows:
\begin{align} \label{triangle_d1}
\xymatrix{
 y \ar[r]^-g &
  x_1 \ar[r]^-{d_1} &
   x_0 \ar[r]^-f &
    \Sigma y.
}
\end{align}
Then $M\cong \Img(f)$, and it only remains to show that $y\in \add \cC$.
If $X\in \cT$ is any object, we must have $\Ext^1_\cC(M, h_\cC X)= 0$ because $M$ is Gorenstein projective and $h_\cC X \in \FD \cC$ by hypothesis. Using \eqref{exact_seq_M} to compute this Ext group, we observe that 
\begin{align*} %\label{Ext1_computed_1}
\ker \big(\cT (x_1, X) \stackrel{d_2^*}{\longrightarrow} \cT (x_2, X) \big) / 
\Img \big(\cT (x_0, X) \stackrel{d_1^*}{\longrightarrow} \cT (x_1, X) \big) 
= 0.
\end{align*}
By applying $\cT(-,X)$ to the triangle \eqref{triangle_d1} we see that $\Img (d_1^*) = \ker(g^*)$, so that we may rewrite the above as follows:
\begin{align} \label{Ext1_computed_2}
\ker \big(\cT (x_1, X) \stackrel{d_2^*}{\longrightarrow} \cT (x_2, X) \big) / 
\ker \big(\cT (x_1, X) \stackrel{g^*}{\longrightarrow} \cT (y, X) \big) 
= 0.
\end{align}
Now complete $d_2$ to a distinguished triangle:
\begin{align}  \label{triangle_d2}
\xymatrix{
 x_2 \ar[r]^-{d_2} &
  x_1 \ar[r]^-{h} &
   z \ar[r] &
    \Sigma x_2.
}
\end{align}
Since $hd_2=0$, by choosing $X:=z$ in \eqref{Ext1_computed_2} we see that $hg=0$. Hence by the triangle \eqref{triangle_d2} there exists a $u: y\to x_2$ such that $d_2u=g$. Also, since $d_1d_2=0$, by the triangle \eqref{triangle_d1} there exists a $v: x_2\to y$ such that $gv=d_2$. 
As $gvu=d_2u=g$ by the definitions of $u$ and~$v$, we have $g(1_y-vu)=0$ and therefore we can find a map $w: y\to \Sigma^{-1}x_0$ such that $(\Sigma^{-1}f)w=1_y-vu$, that is: 
$1_y= (\Sigma^{-1}f)w + vu$. In other words, the composite
\[
\xymatrix{
y \ar[r]^-{\left[ {}^u_w \right]} & x_2 \oplus \Sigma^{-1} x_0 \ar[r]^-{[v \,\, \Sigma^{-1}\!f]} & y
}
\]
is the identity of~$y$, showing that $y$ belongs to $\add \cC$ as claimed.
\end{proof}

\begin{example} \label{ex:C_triangulated}
Let $\cC\subset \cT$ be a small \emph{triangulated} subcategory of~$\cT$. Then $\cC$ is obviously Gorenstein closed in $\cT$ by part~(4) of Proposition~\ref{prop:equiv_Gore_closed}. Since $\cC$ has weak kernels, $\Mod \cC$ is automatically locally coherent. Thus in this case, as soon as we know that $\cC$ is Gorenstein we can deduce from Theorem~\ref{thm:exact_detection} that the $\cC$-module $h_\cC(X)$ arising from any object $X\in \cT$ has a finite projective resolution. In fact $\cC$ could be Gorenstein, for instance, for purely set-theoretic reasons as in Theorem~\ref{thm:triangulated_n_Gorenstein}.
\end{example}

%%%
\section{Universal coefficient theorems}
\label{sec:UCT}

Let $\cT$ be a ($\base$-linear) triangulated category, and let $\cC$ be a small full subcategory of~$\cT$ closed under suspension, $\Sigma \cC=\cC$. 
Then the restricted Yoneda functor
 $h_\cC \colon \cT \longrightarrow \Mod \cC$, 
 $X\mapsto \cT(-,X)|_\cC$,
is a homological functor on~$\cT$. 
Let $\mathcal I := \{f\in \mathrm{Mor}\,\cT\mid h_\cC(f)=0 \}$ denote
the categorical ideal of morphisms killed by~$h_\cC$, and let $\cI(X,Y):=\mathcal I\cap \cT(X,Y)$.
Then by construction we have an exact sequence as follows, natural in $X,Y\in \cT$:
\begin{equation} \label{eq:pre-uct}
\xymatrix@R=15pt{
0\ar[r] & \mathcal I (X,Y) \ar[d]_{\xi} \ar[r] & \cT(X,Y) \ar[r]^-{h_\cC} & \Hom_\cC (h_\cC X, h_\cC Y) \\
& \Ext^1_\cC (h_\cC \Sigma X, h_\cC Y) &&
}
\end{equation}
The above map $\xi$ is defined by associating to every morphism $f\colon X\to Y$ in~$\mathcal I$ the class $\xi(f)$ of the extension $0\to h_\cC Y\to h_\cC Z\to h_\cC\Sigma X\to 0$, where $X\to Y\to Z\to \Sigma X$ is any distinguished triangle containing~$f$.  

\begin{defi} \label{defi:uct}
We say that \emph{the universal coefficient theorem \textup(UCT, for short\textup) with respect to $\cC$ holds for $X$ and $Y$} if, in the diagram \eqref{eq:pre-uct}, the map $h_\cC$ is surjective and $\xi$ is invertible, so that they provide us with an exact sequence
\[
\xymatrix{
0\ar[r] &
  \Ext^1_{\cC}(h_\cC\Sigma X, h_\cC Y) \ar[r] &
   \mathcal T (X,Y) \ar[r] &
    \Hom_{\cC} (h_\cC X,h_\cC Y) \ar[r] & 0
} .
\]
\end{defi}

Fix $\aleph$ an infinite regular cardinal, or let $\aleph$ be the `cardinality' of a proper class.

\begin{terminology} \label{ter:cardinals}
A set is \emph{$\aleph$-small} if its cardinality is strictly less than~$\aleph$. 
A coproduct is  \emph{$\aleph$-small} (or is an \emph{$\aleph$-coproduct}) if it is indexed by an $\aleph$-small set. 
In an abelian category $\cA$ with a set of finite projective generators (such as $\cA=\Mod \cC$), an object is \emph{$\aleph$-generated} if it admits an epimorphism from an $\aleph$-coproduct of such generators; similarly, an object is \emph{$\aleph$-presented} if it admits a presentation by $\aleph$-coproducts of such generators. 
If $\cS\subseteq \cT$ is a family of objects in a triangulated category with arbitrary $\aleph$-coproducts, we denote by $\Loc_\aleph(\cS)$ the \emph{$\aleph$-localizing subcategory} it generates, i.e., the smallest thick subcategory of $\cT$ containing $\cS$ and closed under $\aleph$-coproducts.
 Note that, in the case where $\aleph$ is the cardinality of a proper class, all these notions reduce to the usual ones, where $\aleph$ is dropped from the notation.
An object $C$ of an additive category is \emph{compact} if $\Hom(C,-)$ preserves all the coproducts that are present in that category. 
\end{terminology}

The next theorem is a general form of the UCT. The relevant notation and terminology can be found in Def.\ \ref{defi:uct} and Term.\,\ref{ter:cardinals}.

\begin{thm} [UCT, local version] \label{thm:uct}
Let $\cT$ be an idempotent complete triangulated category admitting arbitrary $\aleph$-small coproducts, and let $\cC$ be a suspension closed full subcategory of compact objects in~$\cT$. 
Suppose moreover that $h_\cC(X)$ is $\aleph$-generated for every object $X$ of $\Loc_\aleph(\cC)\subseteq \cT$.
Then the UCT holds with respect to $\cC$ for all $X\in \Loc_\aleph(\cC)$ such that  $\pdim_\cC h_\cC(X)\leq 1$ and for all $Y\in \cT$.
\end{thm}

\begin{cor}[UCT, global version] \label{cor:uct}
Let $\cT$ and $\cC$ be as in Theorem \ref{thm:uct} with the further assumption that $h_\cC X$ is $\aleph$-generated and $\pdim_\cC h_\cC X\leq 1$ for all $X\in\cT$.
Then the essential image of $h_\cC$ is a hereditary exact category with the following description:
\[
\cE:= \{ M\in \Mod \cC \mid \pdim\!{}_\cC M\leq 1 \textrm{ and } M \textrm{ is } \aleph\textrm{-presentable} \} \,.
\]
\item
Moreover, the functor $h_\cC\colon \Loc_\aleph(\cC)\to \cE$ is full, essentially surjective, and reflects isomorphisms. In particular it induces a bijection on isomorphism classes of objects.
\end{cor}

Several versions of these results are already in the literature.
Nonetheless, we found no statement with the required generality that we could conveniently cite, so we will provide our own proof below; it ends up being completely self-contained.
Of course we don't claim any originality, as most of the arguments, in some form or another, are well-known and can be found e.g.\ in \cite{christensen:ideals}, \cite{beligiannis:relative} and~\cite{meyernest_hom}. 

For the next few lemmas, assume only that $\cT$ and $\cC$ are as in the first sentence of the theorem.
The first lemma we prove is an extension of Yoneda.

\begin{lemma} \label{lemma:extended_yoneda}
Let $P$ be an $\aleph$-generated projective $\cC$-module. 
Then there exists an (up to isomorphism, unique) $C\in \cT$ which is a retract of an $\aleph$-coproduct of objects of $\cC$ and such that $h_\cC C\cong P$. Moreover, $h_\cC$ induces an isomorphism $\cT(C,X)\cong \Hom_\cC (h_\cC C, h_\cC X)$ and thus $\cT(C,X)\cong \Hom_\cC (P, h_\cC X)$ for all $X\in \cT$.
\end{lemma}

\begin{proof}
If $P$ is representable, this is simply the Yoneda lemma.
Since $\cT$ admits $\aleph$-coproducts and $h_\cC$ preserves them (because the objects of $\cC$ are compact), we may extend the result to $\aleph$-coproducts of representables. In particular, we see that $h_\cC$ restricts to an equivalence between $\aleph$-coproducts of objects of $\cC$ and $\aleph$-generated free $\cC$-modules. Being an equivalence, it induces bijections of idempotent morphisms. Since $\cT$ has split idempotents, this allows us to further extend the result to arbitrary retracts.
\end{proof}

\begin{lemma} \label{lemma:aleph_gen}
The following are equivalent for any $X\in \cT$:
\begin{enumerate}
\item The $\cC$-module $h_\cC X$ is $\aleph$-generated.
\item There exists an exact triangle of the form
$K \to C \to X \to \Sigma K$, 
where $h_\cC(X\to \Sigma K)=0$ and $C=\coprod_{i\in I} C_i$ with $C_i\in \cC$ and $|I|<\aleph$.
\end{enumerate}
In particular, if $h_\cC X$ is $\aleph$-generated for all objects $X\in \Loc_\aleph(\cC)$ (or for all $X\in \cT$) then in fact all such $h_\cC X$ are also
 $\aleph$-presentable.  
\end{lemma}

\begin{proof}
Statement (1) means that there exists an epimorphism $\coprod_{i\in I} h_\cC C_i\to h_\cC X$, with $I$ and the $C_i$ as in~(2). By (extended) Yoneda this corresponds to a map $\coprod_iC_i\to X$, which can be completed to a triangle as required. The other direction is similar: the map $C\to X$ in the triangle corresponds as above to a morphism $\coprod_{i\in I}h_\cC C_i \to X$ which is an epimorphism since $h_\cC$ is homological.

For the last claim, note first that if in (2) the object $X$ belongs to $\Loc_\aleph(\cC)$ then also $K$ does; and if in (2) the module $h_\cC K$ is also $\aleph$-generated we may use $(1)\Rightarrow (2)$ to obtain a second triangle and splice the exact sequences obtained after applying $h_\cC$ to exhibit an $\aleph$-presentation of~$X$.
\end{proof}

\begin{lemma} \label{lemma:Add}
If $X\in \Loc_\aleph(\cC)$ and $h_\cC X$ is $\aleph$-generated and projective, then $X$ must be a retract of an $\aleph$-coproduct of objects of~$\cC$.
\end{lemma}

\begin{proof}
If $h_\cC X$ is projective, then by the first part of Lemma \ref{lemma:extended_yoneda} we find a retract $C$ of an $\aleph$-coproduct of objects of~$\cC$ -- so that in particular $C\in \Loc_\aleph(\cC)$ -- and an isomorphism $\varphi\colon h_\cC C\smash{\stackrel{\sim}{\to}} h_\cC X$. By the second part of the lemma there is a morphism $f\colon C\to X$ such that $h_\cC(f)=\varphi$. 
Hence $h_\cC \cone(f)=0$, but if $X\in \Loc_\aleph(\cC)$ then also $\cone(f)\in \Loc_\aleph(\cC)$ and therefore  $\cone(f)=0$, i.e., $f\colon C\to X$ is invertible.
\end{proof}

\begin{lemma} \label{lemma:res_triangle}
Assume now that $h_\cC(\Loc_\aleph(\cC))$ consists of $\aleph$-generated modules. 
Then for any $X\in \Loc_\aleph (\cC)$ with $\pdim_\cC h_\cC X\leq 1$ there exists an exact triangle $C_1\to C_0\to X\to \Sigma C_1$ where 
$h_\cC(X\to \Sigma C_1)=0$, the object $C_0$ is an $\aleph$-coproduct of objects in~$\cC$, and $C_1$ is a retract of such a coproduct.
\end{lemma}

\begin{proof}
Let $K\to C\to X\to \Sigma K$ be a triangle as in Lemma \ref{lemma:aleph_gen}. 
By hypothesis $h_\cC X$ has projective dimension at most one, hence we know from the exact sequence $0\to h_\cC K\to h_\cC C\to h_\cC X \to 0$ that $h_\cC K$ is projective. 
Note that $X\in \Loc_\aleph(\cC)$ implies $K\in \Loc_\aleph(\cC)$, and so, since $h_\cC K$ is $\aleph$-generated by hypothesis, Lemma \ref{lemma:Add} applied to $h_\cC K$ tells us that $K$ is a retract of an $\aleph$-coproduct of objects of~$\cC$.
\end{proof}

\begin{lemma} \label{lemma:I-proj}
If $C$ is a retract of a coproduct of objects of $\cC$ and $X\to Y$ is a map with $h_\cC (X\to Y)=0$ then any composite of the form $C\to X\to Y$ is zero.
\end{lemma}
\begin{proof}
For any fixed such map $X\to Y$, the class of objects $C$ having this property contains $\cC$ by hypothesis and is clearly closed under coproducts and retracts.
\end{proof}

\begin{proof}[Proof of Theorem~\ref{thm:uct}]
Let $X\in \cT$ be such that $X\in \Loc_\aleph (\cC)$ and $\pdim_\cC h_\cC X\leq 1$, and let $Y$ be any object.
Choose a triangle as in Lemma~\ref{lemma:res_triangle} and consider also its image in $\Mod \cC$ (both pictured in the first row):
\[
\xymatrix@C=22pt{
C_1 \ar[r]^-u & C_0 \ar[r]^-v \ar[dr]_g & X \ar[r]^-i \ar@{..>}[d]^f &\Sigma C_1  & 0 \ar[r] & h_\cC C_1 \ar[r] & h_\cC C_0 \ar[dr]_\psi \ar[r]^-\pi & h_\cC X \ar[r] \ar[d]^{\varphi} &0 \\
&& Y &&&&& h_\cC Y &
}
\]
Now let $\varphi\colon h_\cC X\to h_\cC Y$ be any morphism in $\Mod \cC$; we must show that it lifts to~$\cT$.
Under Yoneda, the map $\psi:= \varphi \pi$ corresponds to a (unique) morphism $g\colon C_0\to Y$. 
We have $h_\cC(gu)=0$ and thus also $gu=0$ by Lemma \ref{lemma:I-proj}. Hence $g$ factors through $C_0\to X$ via a map $f\colon X\to Y$. 
Because  $\pi$ is epimorphic and $h_\cC(f) \circ \pi = \psi = \varphi \pi$, we have $h_\cC(f)=\varphi$. In conclusion, we have proved surjectivity of the rightmost map in the UCT exact sequence.

It remains to show that for such $X$ and $Y$ the map $\xi$ is invertible. 
First note that $\cI(X,Y)$ is the kernel of $v^*\colon \cT(X,Y) \to \cT( C_0,Y )$.
Indeed, let $f\in \cT(X,Y)$; if $f\in \cI$ then $fv=0$ by Lemma~\ref{lemma:I-proj}, and conversely if $fv=0$ then $f$ factors through $i\in \cI$ and so belongs to the ideal~$\cI$.
We thus obtain the following commutative diagram in $\Mod \base$, where the three-term middle row is exact and obtained by applying $\cT(-,Y)$ to the triangle $(u,v,i)$:
\begin{equation*} %\label{eq:uct_proof}
\xymatrix@R=6pt@C=16pt{
&& 0 \ar[dr] & f \ar@{}[d]|{\rotatebox{-90}{$\in$}} & 0 \\
&& {\tilde f} \ar@{}[d]|{\rotatebox{-90}{$\in$}} & \cI(X,Y) \ar[ur] \ar@{..>}[dd]^<<<\simeq \ar[dr] & \\
\cT (\Sigma C_0, Y )  \ar[dd]_\simeq \ar[rr]_-{\Sigma u^*} &&
 \cT (\Sigma C_1, Y)  \ar[ur] \ar[rr]_<<<<<{i^*} \ar[dd]_\simeq &&
   \cT(X,Y)   \\
&&&
 \Ext^1(h_\cC \Sigma X,h_\cC Y) &  \\
\Hom(h_\cC\Sigma C_0, h_\cC Y) \ar[rr] &&
 \Hom(h_\cC \Sigma C_1, h_\cC Y) \ar[ur]_\partial &&
}
\end{equation*}
Here the two (solid) vertical isomorphisms are by Yoneda. 
Because of the Ext long exact sequence in $\Mod \cC$ and the projectivity of~$h_\cC \Sigma C_0$ they induce the (dotted) isomorphism $\cI(X,Y)\to \Ext^1_\cC(h_\cC \Sigma X,h_\cC Y)$.
It only remains to verify that the latter coincides with the map~$\xi$. 

In order to apply $\xi$ to an $f\in \cI(X,Y)$ we must choose an exact triangle $X\to Y\to Z\to \Sigma X$ containing~$f$, apply $h_\cC$ to it, and take the Ext class of the resulting short exact sequence. 
But by Lemma \ref{lemma:I-proj} again, $f$ must factor through $i$ and we may then complete to a morphism of triangles:
\[
\xymatrix{
X  \ar[r]^-i \ar@{=}[d] & \Sigma C_1 \ar[d]^-{\tilde f} \ar[r] & \Sigma C_0 \ar[d] \ar[r] & \Sigma X \ar@{=}[d] \\
X \ar[r]^-f & Y \ar[r] & Z \ar[r] & \Sigma X
}
\]
Under $h_\cC$ this becomes the following morphism of short exact sequences, where the left square must then be a pushout:
\[
\xymatrix{
0  \ar[r]  & h_\cC \Sigma C_1 \ar[d]^{h_\cC(\tilde f)} \ar[r] & h_\cC \Sigma C_0 \ar[d] \ar[r] & h_\cC \Sigma X \ar@{=}[d] \\
0 \ar[r] & h_\cC Y \ar[r] & h_\cC Z \ar[r] & h_\cC \Sigma X
}
\]
In other words, we have chosen a map $\tilde f\colon \Sigma C_1\to Y$ such that  $i^*(\tilde f)=f$ and shown that applying $h_\cC$ and the boundary~$\partial$ to it give the same result. Thus the dotted map is indeed~$\xi$, as claimed.
\end{proof}

\begin{proof}[Proof of Corollary \ref{cor:uct}]
By hypothesis $h_\cC(\cT)$ consists of $\aleph$-generated modules, hence of $\aleph$-presentable modules by Lemma~\ref{lemma:aleph_gen}.
Note that both conditions for a module being in $\cE$ are inherited by extensions, so $\cE$ is an exact subcategory of~$\Mod \cC$. 
If $M\in \cE$, by definition we may choose a projective resolution of the form $0\to P_1\to P_0\to M\to 0$, where the $P_i$ are retracts of $\aleph$-coproducts of representables. 
By Lemma~\ref{lemma:extended_yoneda}, we may find a morphism $C_1\to C_0$ in $\cT$ such that $h_\cC(C_1\to C_0)=(P_1\to P_0)$ and where the $C_i$ are corresponding retracts of $\aleph$-coproducts of objects of~$\cC$. 
Let $C_1\to C_0 \to X\to \Sigma C_1$ be an exact triangle containing this map. Then $h_\cC(X\to \Sigma C_1)=0$,  $X\in \Loc_\aleph(\cC)$ and, by comparing cokernels, $h_\cC X\cong M$.
It follows that $\cE$ is contained in the image of $h_\cC\colon \Loc_\aleph(\cC)\to \Mod \cC$, and also that the latter coincides with the image of $h_\cC\colon \cT \to \Mod \cC$ (choose $M:=h_\cC Y$ for an arbitrary $Y\in \cT$).
Lemma~\ref{lemma:res_triangle} implies the reverse inclusion. 
Theorem \ref{thm:uct} shows that $h_\cC\colon \Loc_\aleph(\cC)\to \cE$ is full, hence it only remains to show that it reflects isomorphisms.
So let $f\colon X\to Y$ be a morphism in $\Loc_\aleph (\cC)$ with $h_\cC(f)$ invertible.
In other words, $h_\cC \cone(f) =0$. Since $\cone(f)\in \Loc_\aleph (\cC)$ we already have $\cone(f)=0$, and therefore $f$ is invertible.
\end{proof}

\begin{example}[The classical UCT] \label{ex:classical_uct}
Let $R$ be a dg~algebra whose cohomology ring $H^*R$ is hereditary and consider $\cT:=\D(R)$, its unbounded derived category  of dg~modules. 
(For instance, $R$ could be a hereditary ring concentrated in degree zero, so that $\D(R)$ is its usual derived category of complexes.)
Let $\cC:=\{\Sigma^n R\mid n\in \Z\}$. 
Then $\cC$ is a suspension closed set of compact generators and $h_\cC$ is the cohomology functor $H^*\colon \D(R)\to \Mod H^*R$. 
With $\aleph$ the cardinality of a proper class, Corollary \ref{cor:uct} yields the familiar UCT of classical homological algebra.
\end{example}

\begin{example}[The Rosenberg-Schochet UCT] \label{ex:classical_uct_kk}
Slightly less classical, but also immensely useful, is the following version of the UCT due to Rosenberg and Schochet~\cite{rs}.
Let $\cT:=\KK$ be the Kasparov category, whose objects are separable complex C*-algebras and whose Hom groups are Kasparov's bivariant K-theory groups (see \cite{kasparov_first}~\cite{meyer_cat}). It is a tensor triangulated category admitting arbitrary countable coproducts and such that $\KK(\mathbb C, \Sigma^* \mathbb C)\cong \mathbb Z[\beta, \beta^{-1}]$ with $\mathbb Z$ in degree zero and an invertible `Bott' element $\beta: \mathbb C\stackrel{\sim}{\to}\Sigma^2\mathbb C$ of degree two. Since $\mathbb C$ is the unit object for the tensor structure, $\beta$ induces a periodicity isomorphism $\Sigma^2\cong \id_\KK$. Set $\aleph:= \aleph_1$ and $\cC:=\{\mathbb C, \Sigma\mathbb C\}$. Then $\KK(\Sigma^*\mathbb C,-)=K_*$ is topological K-theory, which yields countable groups on all separable C*-algebras. Thus for $A\in \cB:=\Loc_{\aleph_1}(\mathbb C)\subset \KK$ we obtain the natural UCT short exact sequence
\begin{equation*} %\label{uct}
\xymatrix{ 0 \ar[r] & \Ext^1_\Z(K_{*-1}A, K_*B)\ar[r] & \KK(A,B)\ar[r] & \Hom(K_*A, K_*B) \ar[r]& 0}
\end{equation*}
familiar to operator algebraists (most often in its $\Z/2$-graded version), as well as all its standard corollaries, as special cases of the local UCT Theorem~\ref{thm:uct}. The subcategory $\cB$ is known as the Bootstrap class. Note that for many applications it is quite important to allow algebras~$B$ lying outside of the Bootstrap class.
\end{example}

\begin{remark}
When $\aleph$ is the cardinality of a proper class, as in Example~\ref{ex:classical_uct} or Theorem~\ref{thm:brown-adams} below, then of course every $\cC$-module is $\aleph$-generated and $\aleph$-presented. 
The next proposition shows that if $\Mod \cC$ is locally coherent we can equally well ignore these hypotheses in the local UCT, for any~$\aleph$.
Recall that $\Mod \cC$ is locally coherent if and only if the additive hull of $\cC$ in $\cT$ has weak kernels, which holds for instance if $\cC$ is itself a triangulated category.
\end{remark}

\begin{prop} \label{prop:ln_simplification}
Let $\cT$ be an idempotent complete triangulated category admitting arbitrary $\aleph$-small coproducts. 
Let $\cC$ be a suspension closed full subcategory of compact objects in~$\cT$, such that $\Mod \cC$ is locally coherent.
Then  the $\cC$-module $h_\cC X$ is $\aleph$-presentable for every $X$ in $\Loc_\aleph(\cC)$.
\end{prop}

The proof of the proposition requires the following lemma.

\begin{lemma} \label{lemma:ln_simplification}
If $\Mod \cC$ is locally coherent, then its subcategory of $\aleph $-presentable modules is abelian.
\end{lemma}

\begin{proof}
By a standard reduction it suffices to show that whenever $K$ is an $\aleph$-generated submodule of a coproduct of representables then $K$ is $\aleph$-presentable. 
Thus let $K\subseteq \coprod_{i\in I}C_i$ where the $C_i$ are representable modules. 
If $K$ is $\aleph$-generated it can be written as a directed union $K=\bigcup_{j\in J} K_j$ of finitely generated submodules $K_j$, where $|J|<\aleph$.
Consider for every $j\in J$ the composite inclusion $K_j\to K\to \coprod_iC_i$. 
Since $K_j$ is finitely generated, we can find a finite subset $I(j)\subseteq I$ such that $\coprod_{i\in I(j)}C_i$ still contains~$K_j$.
Hence $K_j$ is finitely presented by the local coherence of $\Mod \cC$.
We conclude that $K= \bigcup_{j\in J} K_j$ is an $\aleph$-indexed directed union of finitely presented modules and is therefore $\aleph$-presented.
\end{proof}

\begin{proof}[Proof of Proposition~\ref{prop:ln_simplification}]
The class of those $X\in \cT$ such that $h_\cC X$ is $\aleph$-presentable contains the objects of~$\cC$ and is closed under $\aleph$-small coproducts.
Since $\Sigma \cC =\cC$, it is also closed under suspension, as $\Sigma$ extends to an autoequivalence of $\Mod \cC$.
If we can verify that it is closed under taking cones, then it must contain all of $\Loc_\aleph(\cC)$ and we are done. 
Let $X\to Y\to Z \to \Sigma X$ be an exact triangle such that $h_\cC X$ and $h_\cC Y$ are $\aleph$-presentable. Then we have an exact sequence 
\[
\xymatrix@R=10pt@C=14pt{ 
h_\cC X \ar[rr] && h_\cC Y \ar[rr] \ar[dr] && h_\cC Z \ar[rr] \ar[dr] && h_\cC \Sigma X \ar[rr] && h_\cC \Sigma Y \\
&& & M \ar[ur] & & N \ar[ur] & &&  }
\]
where $M:= \coker (h_\cC X\to h_\cC Y)$ and $N:=\ker (h_\cC \Sigma X\to h_\cC \Sigma Y)$ are  $\aleph$-presentable by Lemma~\ref{lemma:ln_simplification}. 
The extension $0\to M\to h_\cC Z \to N\to 0$ then shows that $h_\cC Z$ is $\aleph$-presentable, as claimed.
\end{proof}

\begin{center} *** \end{center}

The next theorem connects the UCT with our previous work on Gorenstein homological algebra.
 It is one of the main results of this paper and can be viewed as a generalization of the Brown-Adams representability theorem, as explained below. 

\begin{thm} \label{thm:gore_uct}
Let $\cT$ be an idempotent complete triangulated category with $\aleph$-small coproducts, for some infinite cardinal~$\aleph$. Let $\cC$ be a Gorenstein closed (Def.\,\ref{defi:Gore_closed}) 
and suspension closed full subcategory of compact objects of $\cT$ such that $\Mod \cC$ is locally coherent and 1-Gorenstein.
Then the UCT with respect to $\cC$ holds for all pairs of objects $X,Y\in \cT$ with $X\in \Loc_\aleph(\cC)$. 

Moreover, we have the following dichotomy for an $\aleph$-presented $\cC$-module~$M$:
\begin{itemize}
\item
either $\pdim_\cC M \leq 1$ and $M\cong h_\cC X$ for some $X\in \Loc_\aleph(\cC)\subseteq \cT$,
\item
or $\pdim_\cC M=\infty$ and $M$ is not of the form $h_\cC X$ for any object $X\in\cT$.
\end{itemize}
\end{thm}

\begin{proof}
By Theorem \ref{thm:exact_detection} the functor $h_\cC$ takes values in $\FD \cC$. 
Since the Gorenstein dimension of $\cC$ is one, we thus have $\pdim_\cC h_\cC X\leq 1$ for all $X\in \cT$; since $\Mod \cC$ is locally coherent, Proposition~\ref{prop:ln_simplification} tells us that $h_\cC(\Loc_\aleph(\cC))$ consists of $\aleph$-presentable $\cC$-modules. Hence the first part of the theorem follows from the local UCT, Theorem~\ref{thm:uct}. 
For an arbitrary $\aleph$-presented module~$M$, the claimed dichotomy holds because either $\pdim_\cC M \le 1$ and then $M\cong h_\cC X$ for some $X\in \Loc_\aleph(\cC)$ by the same argument as in Lemma~\ref{lemma:res_triangle}, or $\pdim_\cC M=\infty$ and then we have no chance to represent $M$ by any $X \in \cT$ by Theorem~\ref{thm:exact_detection}.
\end{proof}

In Examples~\ref{ex:classical_uct} and~\ref{ex:classical_uct_kk}, the subcategory $\cC\subset \cT$ is generated under suspensions by a single compact object and is therefore as small as it gets.
At the other extreme, we can choose $\cC$ to consist of \emph{all} compact objects. With this choice we can formally deduce the classical Brown-Adams representability theorem, more precisely Neeman's general version~\cite{neeman:adams_brown}, as an example of a UCT as in Theorem~\ref{thm:gore_uct}. 
The original result for the stable homotopy category then follows by specializing to Example~\ref{ex:spectra}.

\begin{thm}[Brown-Adams representability] \label{thm:brown-adams}
Let $\cT$ be a triangulated category with arbitrary coproducts and such that its category of compact objects, $\cT^c$, admits a skeleton with only countably many morphisms.
Then all cohomological functors on $\cT^c$ are represented by objects in~$\cT$, and all natural transformations between them can be lifted to morphisms in~$\cT$. That is, every natural transformation $H\to H'$ between cohomological functors $H,H'\colon (\cT^c)^\op \to \Mod \base$ is isomorphic to one of the form $\cT(-,X)|_{\cT^c}\to \cT(-,X')|_{\cT^c}$ for some morphism $f\colon X\to X'$ of~$\cT$.
\end{thm}

\begin{proof}
Let $\cC= \cT^c$ and let $\aleph$ be the cardinality of a proper class. Then  $\cC$ is 1-Gorenstein by Theorem~\ref{thm:triangulated_n_Gorenstein}, and $\cC\subset \cT$ is Gorenstein closed and $\Mod \cC$ is locally coherent by Example~\ref{ex:C_triangulated}. We may therefore conclude with Theorem~\ref{thm:gore_uct}. 
\end{proof}

We note the following variant of Brown-Adams representability.

\begin{thm} \label{thm:countable_brown-adams}
Let $\cT$ be a triangulated category with all countable coproducts, such that $\cT^c$ is essentially small and $\cT = \Loc_{\aleph_1}(\cT^c)$. 
Then all $\aleph_1$-generated cohomological functors on $\cT^c$ and all natural transformations between them are representable in~$\cT$.
\end{thm}

\begin{proof}
Since $\cT=\Loc_{\aleph_1}(\cT^c)$ by hypothesis, it follows from Proposition \ref{prop:ln_simplification} (with $\aleph=\aleph_1$ and $\cC=\cT^c$) that $h_{\cT^c} X$ is countably presentable for every object $X\in \cT$. Since each $H=h_{\cT^c} X$ is also a cohomological functor, it follows from the next lemma that it can have projective dimension at most one. Hence the result follows by the global UCT, Corollary \ref{cor:uct}, for $\aleph=\aleph_1$. 
\end{proof}

\begin{lemma} \label{lemma:ctly_hml_flat}
If $H\in \Mod \cT^c$ is a countably presentable cohomological functor then $\pdim_{\cT^c} H\leq 1$.
\end{lemma}

\begin{proof}
Recall that the cohomological functors $H\colon (\cT^c)^\op\to \Mod \base$ are precisely the filtered colimits of representable $\cT^c$-modules (Lemma~\ref{lemma:equiv_modules_triangulated}).
If $H$ is moreover countably presentable we may find a cofinal subsystem which is countable and therefore 
contains a cofinal chain (\cite{kashiwara-schapira:cs}*{Ex.\,3.8}). 
Thus 
$H$ is the colimit of a countable sequence of representables, say 
$H = \colim_{n\in \N}(h_\cC C_0 \smash{\stackrel{a_0}{\longrightarrow}} h_\cC  C_1 \smash{\stackrel{a_1}{\longrightarrow}} \ldots)$. 
We can reassemble the colimit into a short exact sequence
$0\to \coprod_{n\in \N} h_\cC C_n\to \coprod_{n\in \N} h_\cC C_n \to H\to 0$
in $\Mod \cC$, where the components of the first map are $\id_{h_\cC C_{n}}- a_n$. This proves the claim.
\end{proof}

We conclude with an illustrative example. The next and final section of this article contains more substantial examples and applications.

\begin{example} \label{ex:ringel-zhang}
Let $R$ be an $n$-Gorenstein noetherian ring and fix $\pi\in \N$. 
Consider the category $\Ch_\pi(R)$ of $\pi$-periodic complexes of right $R$-modules (Example~\ref{ex:complexes}).
By Proposition \ref{prop:Gorenstein_cplx} we know that $\Ch_\pi(R)$ is an $n$-Gorenstein category. 
In particular $\GProj (\Ch_\pi(R))$ is Frobenius and its stable category $\cT:=\underline{\GProj} (\Ch_\pi(R)) $ is triangulated (Proposition~\ref{prop:triangulated}).
Since $R$ is noetherian $\Ch_\pi (R)$ is locally noetherian and so the triangulated subcategory $\cT^c$ of compact objects is $\underline{\Gproj}(\Ch_\pi (R))$  (Theorem~\ref{thm:triangulated_small}).

For instance, $R=\base Q$ could be the path algebra of a finite acyclic quiver~$Q$  over a field~$\base$, in which case $n=1$ (this is a hereditary finite dimensional algebra), and we may take $\pi=1$, in which case we are studying differential $R$-modules or, equivalently, modules over $R[d]/(d^2)$. 
The latter special case was studied in detail in~\cite{ringel-zhang:dual}. 
Since $R[d]/(d^2)\cong (\base[d]/(d^2))\otimes \base Q$, such modules were viewed in \emph{loc.\,cit. }as ``quiver representations over the dual numbers''.

Now choose the full subcategory $\cC:= \{ R[i] \mid i\in \Z/\pi \}\subseteq \cT$, where $R[i]$ denotes the complex with $R$ in degree~$i$ and zero elsewhere, which is easily seen to be Gorenstein projective. 
Then $\Mod \cC= \prod_{\Z/\pi} \Mod R$ is the category of $\Z/\pi$-graded $R$-modules, which is again $n$-Gorenstein, and $h_\cC\colon \cT \to \Mod \cC$ is simply the cohomology functor $H^*\colon X\mapsto (H^iX)_i$. 
The additive hull of $\cC$ in $\cT^c$ is
\[
\add \cC = 
\{ \oplus_{\alpha\in A} P_\alpha [i_\alpha] \mid P_\alpha\in \proj R, i_\alpha \in \Z/\pi\Z, |A|<\infty \} \,,
\]
from which it follows that every distinguished triangle in $\cT$ with vertices in $\add \cC$ must split. Hence the syzygies in $\Mod \cC$ of such triangles are precisely the finite projectives, so that $\cC$ is Gorenstein closed precisely when $\Gproj \cC = \proj \cC$, that is, when $\gldim\cC<\infty$. In this case the cohomology detects objects and we see that $\cC$ is also a suspension closed set of compact generators.

In particular, for $R$ a hereditary ring Theorem \ref{thm:gore_uct} can be applied (or just use the UCT directly) to get the natural short exact sequence
\[
\xymatrix{
0\ar[r] &
  \Ext^1_R(H^* \Sigma X, H^* Y) \ar[r] &
   \underline{\Hom} (X,Y) \ar[r]^-{H^*} &
    \Hom_R (H^*X,H^*Y) \ar[r] & 0
}
\]
for all $X,Y\in \underline{\GProj} (\Ch_\pi(R))$, and to obtain a bijection between the isomorphism classes of objects in $\underline{\GProj}(\Ch_\pi(R))$ and in $\prod_{\Z/\pi}\Mod R$, where the compact objects, i.e.\ $\stGproj(\Ch_\pi(R))$, correspond to the finitely generated objects of $\prod_{\Z/\pi}\Mod R$ 
(cf.\,\cite{ringel-zhang:dual}*{Theorems~1 and~2}). 
In this case, since $R$ is hereditary, and so in particular of finite global dimension, $\underline{\GProj} (\Ch_\pi(R))$ is the ``derived category of periodic complexes'' and the above result can be viewed as the periodic analogue of the fact that over a hereditary ring a complex is determined up to quasi-isomorphism by its cohomology.
\end{example}

%%%
\section{Examples from KK-theory} \label{sec:KK}

% This is to show the subsections in the TOC (this section is long and structured)
\addtocontents{toc}{\protect\setcounter{tocdepth}{2}}

We have already mentioned in Example~\ref{ex:classical_uct_kk} the Universal Coefficient Theorem of Rosenberg and Schochet, which is the main tool for computing KK-theory groups of C*-algebras. Because of its utility, it has been generalized by many authors and in several directions, often involving variants of KK-theory for C*-algebras with additional structure. These all give rise to triangulated categories satisfying Bott periodicity ($\Sigma^2\cong \id$) and admitting arbitrary countable coproducts.

\subsection{The Universal Multi-Coefficient Theorem} \label{subsec:umct}
Consider the Kasparov category $\KK$ (see Example~\ref{ex:classical_uct_kk}). 
In order to capture finer structure of the KK-theory groups, Dadarlat and Loring~\cite{dadarlat-loring:multi} have produced a variant of the Rosenberg-Schochet UCT called the Universal Multi-Coefficient Theorem, or UMCT. We will now show how to deduce it from our general machinery.

The invariant used in~\cite{dadarlat-loring:multi} is \emph{K-theory with coefficients}, denoted $\underline{K}$. It is obtained as the restricted Yoneda functor $\underline{K}:=h_{\cC}\colon \KK\to \Mod \cC$ for the subcategory
 $\cC := \{\C , \Sigma \C\} \cup \{ \Sigma^i \mathbb I_n \mid n\geq 1 , i\in \{0,1\} \} \subset \KK$, where the objects $\mathbb I_n$ are as in~\cite{dadarlat-loring:multi}. By definition $\mathbb I_n$ is the fiber -- i.e.\ 
 %\Greg{(the desuspension of)} 
 the C*-algebraist's `mapping cone' -- of $n\cdot\id \in \KK(\C,\C)$, so that $\KK(\Sigma^i \mathbb I_n, A)= K_i(A)$ yields K-theory with $\Z/n$-coefficients. 
 (To be precise, rather then $\cC$ the authors of loc.\,cit.\ consider the category $\Lambda$ of `generalized Bockstein operations' acting on the direct sum of all K-theories, with and without coefficients, and have $\underline{K}$ take values in the category of \emph{left} $\Lambda$-modules. But it is clear from the definitions that $\Lambda^\op\cong \cC$ so that we may identify $\underline{K}$ and~$h_{\cC}$.)

\begin{thm} \label{thm:umct_gore}
The category $\cC$ is 1-Gorenstein and is Gorenstein closed in the ambient triangulated category~$\KK$.
\end{thm}

\begin{proof}
Let $\cB= \Loc_{\aleph_1}(\cC)\subset \KK$ denote the Bootstrap category. We claim that $\add \cC= \cB^c$, so that the inclusion $\cC\hookrightarrow \cB^c$ induces an equivalence $\Mod \cB^c \simeq \Mod \cC$.
Indeed this is a corollary of the global form of the Rosenberg-Schochet UCT (see Corollary~\ref{cor:uct}), according to which the K-theory functor $K_*$ induces a bijection between isomorphism classes of objects in $\cB$ and of countable $\Z/2$-graded abelian groups. As $K_*$ commutes with coproducts, the compact objects $A\in \cB^c$ must correspond to finitely generated groups $K_*(A)$ (cf.\,\cite{ivo:boot}*{Lemma~2.9}). 
The claim now follows from the classification of finitely generated abelian groups.

In particular $\add \cC$ is triangulated, hence $\cC$ is trivially Gorenstein closed in~$\KK$ (see Example~\ref{ex:C_triangulated}). Moreover $\cC$ only has countably many morphisms (by easy computations with the Rosenberg-Schochet UCT), hence $\add \cC$ admits a countable skeleton and therefore it is $\leq1$-Gorenstein by Theorem~\ref{thm:triangulated_n_Gorenstein}; it is in fact $1$-Gorenstein, since for instance any $A\in \cB$ with $K_*(A)\cong \mathbb Q$ is not a retract of a coproduct of compact objects (cf.\,Example~\ref{rem:pure_semisimple}). This proves the theorem.
\end{proof}

The Dadarlat-Loring UMCT now follows from Theorems~\ref{thm:umct_gore} and~\ref{thm:gore_uct}.

\begin{remark} \label{rem:KK_equiv}
There is a strong resemblance between the above proof and the Example~\ref{ex:pure_periodic} of 2-periodic complexes of abelian groups; this is not by chance.
Indeed, consider $\cC':=\{\Sigma^i \Z/n \mid n\geq 0, i\in \{0,1\}\}\subset \D_2(\mathbb Z)$ as in Example~\ref{ex:pure_periodic}. There is a suspension-preserving isomorphism of categories $F\colon \cC' \stackrel{\sim}{\to} \cC$ that sends $\Z$ to $\C$ and $\Z/n$ to $\Sigma \mathbb I_n$ ($n\geq1$).
This can be seen immediately by direct inspection from the computations of Example~\ref{ex:pure_periodic} and those performed in~\cite{dadarlat-loring:multi}.
Thus of course $\cC$ is 1-Gorenstein because $\cC'$ is. 
We have seen that $\cC'$ is Gorenstein closed in~$\D_2(\Z)$, and from this we could conclude that $\cC$ is Gorenstein closed in~$\KK$ if we knew that the induced equivalence $\D_2(\Z)^c =\add \cC' \stackrel{\sim}{\to} \add \cC = \cB^c$ preserves distinguished triangles, which is probably the case.

In fact, the above isomorphism $\cC'\cong \cC$ is just a fragment of a suspension-preserving equivalence $F:\D_2(\Z)_{\aleph_1} \stackrel{\sim}{\to} \cB$, where $\D_2(\Z)_{\aleph_1}\subset \D_2(\Z)$ denotes the full subcategory of complexes with countable homology groups. Since every object of $\D_2(\Z)_{\aleph_1}$ is isomorphic to its graded homology, it suffices to define $F$ on complexes concentrated in a single degree and on maps between such objects.
To define $F$ on objects, choose for each countable $M\in \Mod \Z$ a C*-algebra $FM\in \cB$ with $K_0(FM)\cong M$ and $K_1(FM)=0$; then set $F(\Sigma M):=\Sigma (FM)$.
To define $F$ on morphisms, do it separately for maps of degree zero and of degree one by comparing the two UCT exact sequences for $\D_2(\Z)$ and for~$\cB$, which reduce to isomorphisms in both cases.
Clearly $F$ defined this way is a functor (because two maps of degree one always compose to zero)  and an equivalence, but it is less clear how to verify that it preserves triangles (cf.\ \cite{patchkoria}*{Prop.\,5.2.3 and Rem.\,5.2.4}).
\end{remark}

\subsection{Filtrated KK-theory} \label{subsec:filtratedKK}

We now turn to filtrated K-theory, as introduced by Meyer and Nest \cite{meyer-nest:filtrated1}~\cite{meyer-nest:filtrated2}. 
The triangulated category we work with is the Kasparov category $\KK(X)$ of C*-algebras fibered over a finite topological space~$X$. Here we only consider the example where $X$ is the set $\{1,\ldots, n\}$ with $n$ points equipped with the topology where the non-empty open sets are the intervals $[a,n]=\{a,\ldots, n\}$ ($a\in X$). 
An object of $\KK(X)$ is a sequence $A\colon A_1 \supseteq \ldots \supseteq A_n \supseteq A_{n+1}=0$ of C*-algebras, one for each open subset, such that $A_a\supseteq A_{a+1}$ is the inclusion of a (two-sided) ideal. Filtrated K-theory $\FK$ assigns to such an object the collection of the ordinary K-theory groups of all subquotients in the sequence, together with the action of all natural transformations. This invariant is representable as the restricted Yoneda functor $h_\cC= \FK$ associated to a certain small category $\cC\subset \KK(X)$ that we now describe by generators and relations, following the computations of~\cite{meyer-nest:filtrated2}*{\S3}. 
(Similarly to Dadarlat and Loring with the UMCT of Section~\ref{subsec:umct}, Meyer and Nest also prefer to work with the $\Z/2$-graded category of natural transformations, denoted $\mathcal{NT_*}$, and left graded modules over it; our description of $\cC$ is the ungraded, contravariant translation of the explicit computation of~$\mathcal{NT_*}$ in~\cite{meyer-nest:filtrated2}.)

The category $\cC$ comprises the following suspension closed set of compact objects:
\[
\obj(\cC) =\{ \Sigma^i R_{[a,b]} \mid i\in \{0,1\} \textrm{ and } a,b\in X, a\leq b \}\,.
\]
(The intervals $[a,b]$ correspond to the non-empty connected locally closed subsets of~$X$. The Hom group $\KK(X)(\Sigma^iR_{[a,b]}, A)=\FK_{[a,b],i}(A)$ yields the $i$-th K-theory group of the subquotient $A_a/A_{b+1}$ of~$A_1$.) 

We know by \cite{meyer-nest:filtrated2}*{(3.1)} that the non-zero Hom groups of $\cC$ are all free abelian groups of rank one, as follows:
\begin{align*}
\cC (R_{[a,b]} , R_{[a',b']}) & = \Z \cdot \mu^{[a,b]}_{[a',b']}
 \quad \textrm{ if } a\leq a' \leq b \leq b',\\
\cC (R_{[a,b]} , \Sigma R_{[a',b']}) & = \Z \cdot \delta^{[a,b]}_{[a',b']} 
 \quad \textrm{ if } a-1\leq b', a'< a \textrm{ and } b'< b. 
\end{align*}
The notations $\mu,\delta$ for the two family of generators match those in~\cite{meyer-nest:filtrated2}*{\S3.2}, where the composites of these generators are also described.

\begin{construction} \label{constr:filtr-KK}
We now introduce a new notation for the objects and arrows of $\cC$ in order to better exploit its symmetries.
Set 
\[
 R_{[a,b]} =: A_{a,b}  \quad \textrm{ and } \quad \Sigma R_{[a,b]}  =: A_{b+1,a+n}  \quad (a,b\in X, a\leq b),
\]
and denote by 
\[\alpha^{a,b}_{a',b'}: A_{a,b}\to A_{a',b'} \] 
either the appropriate Meyer-Nest generator $\mu^{[a,b]}_{[a',b']}$ or $\delta^{[a,b]}_{[a',b']}$, or its suspension (generating the Hom group $\cC(\Sigma R_{[a,b]}, \Sigma R_{[a',b']})$ or $\cC(\Sigma R_{[a,b]}, R_{[a',b']})$), or the zero map (if the corresponding Hom group vanishes).

If we arrange the objects $A_{a,b}$ and all their suspensions at the vertices of a $\Z^2$-lattice, Bott periodicity tells us that we must have $A_{a,b}=\Sigma^2 A_{a,b} = A_{a+n+1, b+n+1}$.

We deduce the following presentation for~$\cC$. The generating quiver is the following doubly infinite grid, which we declare to be periodic with period $(a,b)\mapsto (a+n+1,b+n+1)$: 
\begin{equation} \label{eq:quiver:KK(X)}
\vcenter{
\xymatrix{
& \vdots \ar[d] & \vdots \ar[d] & \\
\cdots \ar[r] & A_{a,b} \ar[d]_{\alpha^{a,b}_{a+1,b}} \ar[r]^{\alpha^{a,b}_{a,b+1}} & A_{a,b+1} \ar[d]^{\alpha^{a,b}_{a+1,b}} \ar[r] & \cdots \\
\cdots \ar[r] & A_{a+1,b} \ar[d] \ar[r]^{\alpha^{a,b}_{a+1,b+1}} & A_{a+1,b} \ar[d]\ar[r] & \cdots \\
&\vdots & \vdots & 
}
}
\end{equation}
The relations are as follows:
\begin{itemize}
\item every square is commutative, and
\item $A_{a,b}=0$ outside of the ``fat diagonal'' $\{(a,b)\in \Z^2 \mid a\leq b\leq a+n-1\}$.
\end{itemize}
\end{construction}

\begin{remark}
We can describe the above presentation in terms of the notations of Meyer-Nest.
As generating quiver we have taken the following one, where instead of $R_{[a,b]}$ we simply write $ab$ for space, and where every depicted arrow is either one of the generators $\mu, \delta$ or its suspension. 
\begin{equation*} %\label{eq:quiver-extended:KK(X)}
\vcenter{
\xymatrix@C=13pt@R=14pt{
\ar[d] & \ar[d] & \ar[d] &&&&&&& \\
11 \ar[r] & 12 \ar[r]\ar[d] & 13 \ar[r]\ar[d] & \cdots \ar[r] & 1n \ar[d] &&&&& \\
& 22 \ar[r] & 23 \ar[r]\ar[d] & \cdots \ar[r] & 2n \ar[r] \ar[d] & \Sigma 11 \ar[d] &&&& \\
&& 33 \ar[r] & \cdots \ar[r] & 3n \ar[r] \ar[d] & \Sigma 12 \ar[d]\ar[r] & \Sigma 22 \ar[d] &&& \\
&&& \ddots & \vdots \ar[d]  & \vdots \ar[d] & \vdots \ar[d] & \ddots && \\
&&&& nn  \ar[r] & \Sigma 1\,n\!\!-\!\!1 \ar[d]\ar[r] & \Sigma 2\,n\!\!-\!\!1 \ar[d]\ar[r] & \cdots \ar[r] & \Sigma n\!\!-\!\!1\,n\!\!-\!\!1\ar[d] & \\
&&&&& \Sigma 1n \ar[r] & \Sigma 2n \ar[r]\ar[d] & \cdots \ar[r] & \Sigma n\!\!-\!\!1\,n \ar[r]\ar[d] & \Sigma nn \ar[d]\\
&&&&&&&&&
}
}
\end{equation*}
This shape forms a closed band by Bott periodicity $\Sigma^2=\id$: the $n-1$ vertical maps sticking out of the bottom row are the same as the $n-1$ maps entering the top row. We thus have a basic triangular shape containing the objects $\{R_{[a,b]} \mid a,b \in X , a\leq b\}$ together with its suspended copy, glued together on two sides by a set of $n-1$ connecting maps and their suspensions.
The relations imposed on the quiver are those saying that every square commutes and any two consecutive maps at the border of the diagram compose to zero.
\end{remark}

\begin{remark} \label{rem:canonical_maps}
Let $A_{a,b}$ be a non-zero object of the quiver~\eqref{eq:quiver:KK(X)}, that is $a\leq b\leq a+n-1$. It follows from the commutativity relations that each of the canonical maps $\alpha^{a,b}_{a',b'}: A_{a,b}\to A_{a',b'}$ is the unique composite map $A_{a,b}\to A_{a',b'}$ that can be produced as a composite of the displayed generators. If it is non-zero, we must have $b \leq b' \leq a+n-1$ and $a\leq a' \leq b$, which means that $(a',b')$ is contained in the following box inside the ``fat diagonal'' which is delimited by the dashed lines:

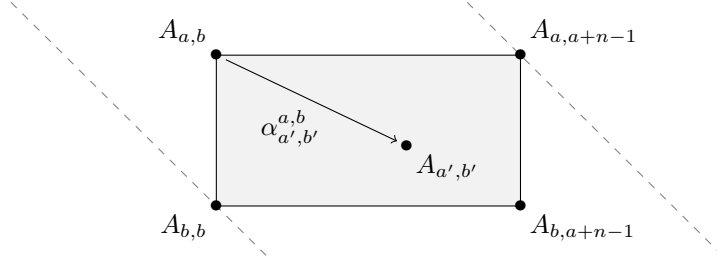
\begin{center}
\begin{tikzpicture}
  % The "fat diagonal"
  \draw[gray,dashed] (-2.7,.7) -- (.7,-2.7);
  \draw[gray,dashed] (3.3,.7) -- (6.7,-2.7);

  % The box with morphisms
  \filldraw[fill=black!5!white, draw=black] (0,0) rectangle (4,-2);
  
  % Objects delimiting the rectangle
  \node at (0,0)  {$\bullet$}; \node[anchor=south east] at (0,0)  {$A_{a,b}$};
  \node at (4,-2) {$\bullet$}; \node[anchor=north west] at (4,-2) {$A_{b,a+n-1}$};
  \node at (0,-2) {$\bullet$}; \node[anchor=north east] at (0,-2) {$A_{b,b}$};
  \node at (4,0)  {$\bullet$}; \node[anchor=south west] at (4,0)  {$A_{a,a+n-1}$};

  % Our guy inside the rectangle
  \node at (2.5,-1.2) {$\bullet$};
  \node[anchor=north west] at (2.5,-1.2) {$A_{a',b'}$};
  \draw[black,->] (0.125,-0.06) -- (2.375,-1.14);
  \node[anchor=north east] at (1.5,-.6) {$\alpha^{a,b}_{a',b'}$};
\end{tikzpicture}
\end{center}
In particular, $\Mod\cC$ is a locally noetherian category since $\cC_c$ is a free abelian group of finite total rank for each $c \in \cC$.
\end{remark}

\begin{remark} \label{rem:filt-KK-tria}
Consider a canonical non-zero map $\alpha^{a,b}_{a,b'}: A_{a,b}\to A_{a,b'}$ as in Remark~\ref{rem:canonical_maps} and such that $a \le b < b' \le a+n-1$ (it is horizontal and between objects of the basic domain $\{R_{[i,j]}\}_{i,j}$). 
This map is part of a distinguished triangle of the form
\begin{equation} \label{eq:filt-KK-tria}
\xymatrix{ A_{a,b} \ar[r]^-{\alpha^{a,b}_{a,b'}} &
 A_{a,b'} \ar[r]^-{\alpha^{a,b'}_{b+1,b'}} &
  A_{b+1,b'} \ar[r]^-{\alpha^{b+1,b'}_{b+1,a+n}} & A_{b+1,a+n}=\Sigma A_{a,b} }
\end{equation}
whose every map is canonical.
(This can be seen by considering, as in \cite{meyer-nest:filtrated2}*{\S3}, the canonical extensions involving the subquotient C*-algebras of a general object of~$\KK(X)$.)
In particular this triangle is contained in~$\cC$.
\end{remark}

Let us state the main result about $\cC \subset \KK(X)$, which again allows us to invoke Theorem~\ref{thm:gore_uct} with $\aleph = \aleph_1$ to obtain the corresponding version of UCT, originally coming from~\cite{meyer-nest:filtrated2}.

\begin{thm} \label{thm:filt-KK-Gore-closed}
Let $\cC$ be the full subcategory representing the filtrated K-theory inside Kasparov's category $\KK(X)$ for $X = \{1,\dots,n\}$, as described in \S\ref{subsec:filtratedKK}. Then $\cC$ is $1$-Gorenstein and Gorenstein closed in $\KK(X)$.
\end{thm}

The first observation is that Remark~\ref{rem:canonical_maps} implies the existence of a Serre functor for $\cC$ (where the base ring is $\base = \Z$).

\begin{lemma} \label{lemma:filt-KK-Serre}
Let $\cC \subset \KK(X)$ for $X = \{1,\dots,n\}$ be as described above. Then $\cC$ admits a Serre functor (in the sense of Def.~\ref{defi:Hom-finite}).
\end{lemma}

\begin{proof}
Let us define the automorphism $S\colon \cC \to \cC$ so that
$S(A_{a,b}) = A_{b,a+n-1}$ on objects, and
$S(\alpha^{a,b}_{a',b'}) = \alpha^{b,a+n-1}_{b',a'+n-1}$, i.e.\ $S$ sends canonical maps to canonical maps in the sense of Remark~\ref{rem:canonical_maps}.
In plain english, $S$ sends an object $c \in \cC$ to the unique ``farthest'' object $d \in \cC$ such that $\Hom_\cC(c,d) \ne 0$. We leave it to the reader to check that $S$ is well-defined.

In order to specify the isomorphism $\sigma_{c,d}$ from Def.~\ref{defi:Hom-finite}, we use Remark~\ref{rem:trace-Serre}. Since $\Hom_\cC(c,Sc)$ is freely generated by $\alpha^c_{Sc}\colon c \to Sc$ over $\Z$, we can simply define $\lambda_c\colon \Hom_\cC(c,Sc) \to \Z$ by putting $\lambda_c(n\cdot \alpha^c_{Sc}) := n$. It is straightforward to check that all the required properties are satisfied.
\end{proof}

The lemma we have just proved together with Theorem~\ref{thm:Serre_Gor} already implies that $\cC$ is $1$-Go\-ren\-stein. In order to prove that it is also Gorenstein closed, we shall now discuss Gorenstein projective $\cC$-modules in more detail.

\begin{lemma} \label{lemma:filt-KK-Gproj}
A right $\cC$-module $M$ belongs to $\Gproj\cC$ \iff $M_c$ is a free abelian group of finite rank for each $c \in \cC$.
\end{lemma}

\begin{proof}
Clearly $\Mod \Z$ is $1$-Gorenstein and $\GProj \Z = \Proj \Z$. Thus a $\cC$-module is Gorenstein projective \iff the underlying abelian group of $M$ is free, by Corollary~\ref{cor:Serre_GProj}. On the other hand, it is easy to see using Remark~\ref{rem:canonical_maps} that $M$ is finitely presented \iff the underlying group of $M$ is finitely presented (equivalently, generated).
\end{proof}

To prove that $\cC$ is Gorenstein closed, we aim to use the criterion of Proposition~\ref{prop:equiv_Gore_closed}(1). To describe a suitable set $\cS \subseteq \Gproj\cC$, we introduce:

\begin{notation} \label{notation:filt-KK-Gproj}
Let $(a,b), (a',b') \in \Z^2$ inside the ``fat diagonal'' be such that $b \le b' \le a+n-1$ and $a \le a' \le b$ (as in Remark~\ref{rem:canonical_maps}). Then we define a $\cC$-module $G^{a,b}_{a',b'}$ so that its evaluation at $(a'',b'') \in \Z^2$ is $\Z$ if $a \le a'' \le a'$ and $b \le b'' < b'$, and it evaluates to zero otherwise. The horizontal and vertical arrows of the quiver~\eqref{eq:quiver:KK(X)} defining $\cC$ act by the identity maps whenever possible and by zero maps otherwise.

Thus, $G^{a,b}_{a',b'}$ can be viewed as a ``characteristic module'' with value $\Z$ of the rectangle

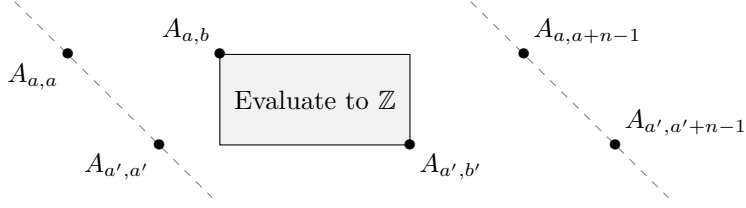
\begin{center}
\begin{tikzpicture}
  % The "fat diagonal"
  \draw[gray,dashed] (-2.7,.7) -- (-.1,-1.9);
  \draw[gray,dashed] (3.3, .7) -- (5.9,-1.9);

  % The box with morphisms
  \filldraw[fill=black!5!white, draw=black] (0,0) rectangle (2.5,-1.2);
  
  % Objects delimiting the rectangle and its label
  \node at (0,0)  {$\bullet$};
  \node[anchor=south east] at (0,0) {$A_{a,b}$};
  \node at (2.5,-1.2) {$\bullet$};
  \node[anchor=north west] at (2.5,-1.2) {$A_{a',b'}$};
  \node at (1.25,-.6) {Evaluate to $\Z$};
  
  % Reference objects at the boundary of the diagonal
  \node at (-2,0)  {$\bullet$};
  \node[anchor=north east] at (-2,0) {$A_{a,a}$};
  \node at (4,0)  {$\bullet$};
  \node[anchor=south west] at (4,0) {$A_{a,a+n-1}$};
  \node at (-0.8,-1.2)  {$\bullet$};
  \node[anchor=north east] at (-0.8,-1.2) {$A_{a',a'}$};
  \node at (5.2,-1.2)  {$\bullet$};
  \node[anchor=south west] at (5.2,-1.2) {$A_{a',a'+n-1}$};
\end{tikzpicture}
\end{center}

\noindent
in the ``fat diagonal'' of $\Z^2$, and $G^{a,b}_{a',b'} \in \Gproj\cC$ by Lemma~\ref{lemma:filt-KK-Gproj}. 
\end{notation}

Now we can explicitly describe the images of the triangles from Remark~\ref{rem:filt-KK-tria}. These must be Gorenstein projective due to Lemma~\ref{lemma:Im_Gore}.

\begin{lemma} \label{lemma:Im-filt-KK}
Let $a \le b < b' \le a+n-1$ and consider the triangle~\eqref{eq:filt-KK-tria}. Then
\[
\Img \alpha^{a,b}_{a,b'} \cong G^{b'-n+1,a}_{a,b}
\quad \textrm{and} \quad
\Img \alpha^{a,b'}_{b+1,b'} \cong G^{b'-n+1,b+1}_{a,b'}.
\]
\end{lemma}

\begin{proof}
By the description of $\cC$ and Remark~\ref{rem:canonical_maps}, we have ${_{A_{a,b}}\cC} \cong G^{b-n+1,a}_{a,b}$ and ${_{A_{a,b'}}\cC} \cong G^{b'-n+1,a}_{a,b'}$. Therefore
\[ \Img \alpha^{a,b}_{a,b'} \cong \Img(G^{b-n+1,a}_{a,b} \to G^{b'-n+1,a}_{a,b'}) \cong G^{b'-n+1,a}_{a,b}, \]
i.e.\ the ``characteristic module'' of the intersection of the two rectangles for ${_{A_{a,b}}\cC}$ and ${_{A_{a,b'}}\cC}$.
The other case is similar.
\end{proof}

Now we exhibit a particular set of modules which generates $\Gproj\cC$ by extensions. It will follow a posteriori that these modules have resolutions given by triangles in $\add\cC \subset \KK(X)$, but \emph{not} by those displayed in Remark~\ref{rem:filt-KK-tria}.

\begin{lemma} \label{lemma:filt-KK-gens}
Each $M \in \Gproj\cC$ possesses a finite filtration
\[ 0 = M_0 \subseteq M_1 \subseteq M_2 \subseteq \cdots \subseteq M_\ell = M \]
such that for each $i$, $M_{i+1}/M_i \cong G^{a_i,b_i}_{a_i,b_i}$ for some $(a_i,b_i) \in \Z^2$ in the ``fat diagonal''.
\end{lemma}

\begin{proof}
We will prove the claim by induction on the rank of the underlying abelian group of $M$, as this makes sense thanks to Lemma~\ref{lemma:filt-KK-Gproj}. The case $M=0$ being trivial, let us fix $0 \ne M \in \Gproj\cC$.

Let $a,b \in \Z$ with $a \le b \le a+n-1$ (i.e.\ $(a,b)$ is a point in the ``fat diagonal'') and consider the map
\[ f_{a,b} := \big(M(\alpha^{a,b-1}_{a,b}), M(\alpha^{a-1,b}_{a,b})\big)^t \colon M_{A_{a,b}} \to M_{A_{a,b-1}} \oplus M_{A_{a-1,b}}. \]
(Here we view $M$ as a functor $M\colon \cC^\op \to \Mod\Z$ and appeal to \eqref{eq:quiver:KK(X)} for the notation for objects and morphism of $\cC$.)
Since $\Z$ is a hereditary ring and both ends of $f_{a,b}$ are free groups, it follows that $\ker f_{a,b}$ is a direct summand of $M_{A_{a,b}}$. In particular, $\ker f_{a,b} \cong \Z^\ell$ for some $\ell\ge 0$ and, since both $\alpha^{a,b-1}_{a,b}$ and $\alpha^{a-1,b}_{a,b}$ act by zero on $\ker f_{a,b}$ by definition, we obtain a short exact sequence of $\cC$-modules with all three terms in $\Gproj\cC$:
\[ 0 \to (G^{a,b}_{a,b})^\ell \to M \to M/(G^{a,b}_{a,b})^\ell \to 0. \]

Observe that to finish the proof, it suffices to show that there exist $a,b \in \Z$ such that $\ell \ge 1$. Indeed, in this case the underlying group of $M/(G^{a,b}_{a,b})^\ell$ has a strictly smaller rank than $M$ and we are done by the inductive hypothesis. Note that we can extend the definition of $f_{a,b}$ to all $a,b \in \Z^2$ and the problem does not change, since $M_{A_{a,b}} = 0$ for all $a,b$ outside the ``fat diagonal'' by virtue of Construction~\ref{constr:filtr-KK}.

To finish the inductive step, pick any $a,b \in \Z^2$ such that $M_{A_{a,b}} \ne 0$. Let us consider the composition of $n$ maps
\begin{multline*}
M_{A_{a,b}} \to
M_{A_{a,b-1}} \oplus M_{A_{a-1,b}} \to
\\
\to (M_{A_{a,b-2}} \oplus M_{A_{a-1,b-1}}) \oplus (M_{A_{a-1,b-1}} \oplus M_{A_{a-2,b}}) \to 
\cdots \\
\cdots \to
\bigoplus_{I \subseteq \{1,\dots,n\}} M_{A_{a-n+\lvert I\rvert,b-\lvert I\rvert}},
\end{multline*}
where each component of each of the maps is given by $f_{a',b'}$ for the corresponding $a',b' \in \Z$. Observe that the composition vanishes. To see that, each of the $2^n$ components $M_{A_{a,b}} \to M_{A_{a-n+i,b-i}}$ of the composition is a path of length $n$ in the quiver~\eqref{eq:quiver:KK(X)}. However, each such path vanishes in $\cC$ because the commutativity relations make it equal to a path passing through a vertex outside the ``fat diagonal'' of $\Z^2$. As an upshot, there must be $f_{a',b'}$ involved which is not injective, and hence $\ker f_{a',b'} \cong \Z^\ell$ with $\ell\ge 1$ for such a pair $a',b'$.
\end{proof}

Finally, we finish the proof of the theorem.

\begin{proof}[Proof of Theorem~\ref{thm:filt-KK-Gore-closed}]
The category $\cC \subset \KK(X)$ is $1$-Gorenstein by Theorem~\ref{thm:Serre_Gor} and Lemma~\ref{lemma:filt-KK-Serre}.

To prove that $\cC$ is Gorenstein closed in $\KK(X)$, it suffices to show that the images of maps $\alpha^{a,b}_{a,b'}$ and $\alpha^{a,b'}_{b+1,b'}$ in the statement of Lemma~\ref{lemma:Im-filt-KK} jointly generate the modules $G^{a,b}_{a,b}$ from Lemma~\ref{lemma:filt-KK-gens} by taking extensions and retracts. The conclusion will then follow by Proposition~\ref{prop:equiv_Gore_closed}(1).

To this end, note that $\Img \alpha^{a,b}_{a,a+n-1} \cong G^{a,a}_{a,b}$ and $\Img \alpha^{a,a+n-1}_{b,a+n-1} \cong G^{a,b}_{a,a+n-1}$ for each $a \in \Z$ and $a \le b \le a+n-1$. Hence $\Img \alpha^{a,a}_{a,a+n-1} \cong G^{a,a}_{a,a}$, $\Img \alpha^{a,a+n-1}_{a+n-1,a+n-1} \cong G^{a+n-1,a+n-1}_{a+n-1,a+n-1}$, and for each $a < b < a+n-1$ we have a short (and in fact non-split) exact sequence of $\cC$-modules
\[ 0 \to G^{a,a}_{a,b} \to G^{a,b}_{a,b} \oplus G^{a,a}_{a,a+n-1} \to G^{a,b}_{a,a+n-1} \to 0. \]
It follows that Remark~\ref{rem:filt-KK-tria} exhibits a sufficient family of triangles.
\end{proof}

\subsection{Equivariant KK-theory} \label{subsec:equivariantKK}

Let $p \in \N$ be a prime number and denote by $C(p)$ the cyclic group with $p$ elements. Then one can define $C(p)$-equivariant KK-theory of C*-algebras: it is a triangulated category $\KK^{C(p)}$ which satisfies a version of Bott periodicity, $\Sigma^2 \cong \id_{\KK^{C(p)}}$, and admits countable coproducts.

In order to obtain a universal coefficient theorem of the form discussed in the previous sections, K\"ohler~\cite{koehler} constructed in his thesis a very specific full subcategory $\cC \subseteq \KK^{C(p)}$ consisting of 3 suspension orbits. Here it is only important that he gave in~\cite{koehler}*{\S11.9} an explicit presentation of this category as a $\Z$-linear category with an involutive automorphism $\Sigma$ (i.e.\ we take $\base = \Z$). The generating quiver of $\cC$ is given by
\begin{equation} \label{eq:eq-KK-quiver}
\vcenter{
\xymatrix{
& A_0 \ar@<.5ex>[r]^-{{_1\alpha_0}} \ar@<.5ex>[dl]^-{{_2\alpha_0}}  \ar@(l,u)[]^-{{_0t_0}} &
A_1 \ar@<.5ex>[l]^-{{_0\alpha_1}} \ar@<.5ex>[dr]^-{\Sigma{_2\alpha_1}} \ar@(r,u)[]_-{{_1s_1}}
\\
A_2 \ar@<.5ex>[ur]^-{{_0\alpha_2}} \ar@<.5ex>[dr]^-{{_1\alpha_2}} \ar@(l,u)[]^-{{_2s_2}} \ar@(l,d)[]_-{{_2t_2}} &&&
\Sigma A_2 \ar@<.5ex>[dl]^-{\Sigma{_0\alpha_2}} \ar@<.5ex>[ul]^-{\Sigma{_1\alpha_2}} \ar@(r,u)[]_-{\Sigma{_2s_2}} \ar@(r,d)[]^-{\Sigma{_2t_2}}
\\
& \Sigma A_1 \ar@<.5ex>[ul]^-{{_2\alpha_1}} \ar@<.5ex>[r]^-{\Sigma{_0\alpha_1}} \ar@(l,d)[]_-{\Sigma{_1s_1}} &
\Sigma A_0 \ar@<.5ex>[l]^-{\Sigma{_1\alpha_0}} \ar@<.5ex>[ur]^-{\Sigma{_2\alpha_0}} \ar@(r,d)[]^-{\Sigma{_0t_0}},
}
}
\end{equation}
and the list of relations is:
\begin{enumerate}
\item[($\rho1$)] the clockwise or counterclockwise compositions of two arrows labeled by $\alpha$ all vanish, i.e.

\begin{itemize}
 \item ${_2\alpha_1} \circ \Sigma{_1\alpha_0} = 0 = \Sigma{_0\alpha_1} \circ {_1\alpha_2}$,
 \item ${_0\alpha_2} \circ {_2\alpha_1} = 0 = {_1\alpha_2} \circ {_2\alpha_0}$,
 \item ${_1\alpha_0} \circ {_0\alpha_2} = 0 = {_2\alpha_0} \circ {_0\alpha_1}$;
\end{itemize}

\item[($\rho2$)] ${_0\alpha_1} \circ {_1\alpha_0} = N({_0t_0})$, ${_1\alpha_0} \circ {_0\alpha_1} = N({_1s_1})$;
\item[($\rho3$)] ${_0\alpha_2} \circ {_2\alpha_0} = \id_{A_0} - {_0t_0}$, ${_2\alpha_0} \circ {_0\alpha_2} = \id_{A_2} - {_2t_2}$;
\item[($\rho4$)] ${_1\alpha_2} \circ {_2\alpha_1} = \id_{A_1} - {_1s_1}$, ${_2\alpha_1} \circ {_1\alpha_2} = \id_{A_2} - {_2s_2}$;
\item[($\rho5$)] $N({_2t_2}) + N(_2s_2) = p$.
\end{enumerate}
Here $N(x)$ stands for the sum $\sum_{i=0}^{p-1} x^i$ and, unlike in~\cite{koehler}, we compose the generating homomorphisms from right to left as if they were usual maps. In particular, the endomorphism rings of objects are
\begin{align*}
\End_\cC(A_0) &= \Z[{_0t_0}]/(\id_{A_0} - {_0t_0}^p), \\
\End_\cC(A_1) &= \Z[{_1s_1}]/(\id_{A_1} - {_1s_1}^p), \\
\End_\cC(A_2) &= \Z[{_2s_2}, {_2t_2}]/\big(N(_2s_2) + N(_2t_2) - p, (\id_{A_2} - {_2s_2})\circ(\id_{A_2} - {_2t_2})\big).
\end{align*}
We shall use Theorem~\ref{thm:Gorenstein} and standard commutative algebra to prove that

\begin{thm} \label{thm:eq-KK-Gore}
The full subcategory $\cC$ of $\KK^{C(p)}$ which is described just above is $1$-Gorenstein.
\end{thm}

\begin{remark} \label{rem:eq-KK-Gore}
It is also possible to prove that $\cC$ is Gorenstein closed in $\KK^{C(p)}$; 
together with the above theorem, this shows that K\"ohler's UCT falls nicely into our general Gorenstein framework.
Indeed, the Gorenstein closedness of $\cC$ must hold by Theorem~\ref{thm:exact_detection} together with K\"ohler's own proof of his UCT. 
Unfortunately, unlike for Theorem~\ref{thm:eq-KK-Gore}, we have not been able to provide a new argument for Gorenstein closedness that would be significantly different from K\"ohler's original one (\cite{koehler}*{\S12}). This means that, unlike for the other examples of this section, we do not have a complete \emph{new} proof of the UCT for $C(p)$-equivariant KK-theory.
\end{remark}

Let us start with criterion (1) of Theorem~\ref{thm:Gorenstein}. In view of Remark~\ref{rem:FD} we need to prove that the explicitly computed (and commutative noetherian) endomorphism rings of the objects of $\cC$ are all $1$-Gorenstein. In particular, we need to relate our homological definition of Gorenstein rings to results in commutative algebra, which go back to Bass~\cite{bass:gorenstein}. Let us first recall some terminology for the convenience of the reader.

A commutative noetherian local ring $R$ with maximal ideal $\mathfrak{m}$ is called \emph{regular} if the Krull dimension of $R$ equals the vector space dimension of $\mathfrak{m}/\mathfrak{m}^2$ over the residue field $R/\mathfrak{m}$ (see~\cite{bruns-herzog}*{\S2.2}). A general commutative noetherian ring $R$ is \emph{regular} if the localization of $R$ at every prime ideal is a regular local ring. If, moreover, the Krull dimension of $R$ is finite (which is true for all finitely generated algebras over a field or over $\Z$, and by Krull's principal ideal theorem~\cite{bruns-herzog}*{A.1--A.3} also whenever $R$ is local), then there is a very convenient homological characterization of regularity.

\begin{prop} \label{prop:regular}
A commutative noetherian ring $R$ of finite Krull dimension is regular if and only if the global dimension of $\Mod R$ is finite. In this case, $\gldim R$ equals the Krull dimension of $R$.
\end{prop}

\begin{proof}
Let us first assume that $R$ is local of Krull dimension $n$ with maximal ideal $\mathfrak{m}$. Then the homological characterization of regularity is exactly the Auslander-Buchsbaum-Serre theorem, \cite{bruns-herzog}*{Theorem 2.2.7}. It is a standard fact that $\gldim R = \pdim_R R/\mathfrak{m}$; see~\cite{bruns-herzog}*{Corollary 1.3.2}. But for $R$ regular, $R/\mathfrak{m}$ has a very explicit free resolution of length $n$, the Koszul complex~\cite{bruns-herzog}*{Corollary 1.6.14(b)}, from which one easily computes that $\Ext^n_R(R/\mathfrak{m},R) \cong R/\mathfrak{m}$. Hence $\gldim R = n$.

If $R$ is general, then one just applies the standard natural isomorphism
\[ \Ext_R(M,N)_\mathfrak{p} \cong \Ext_{R_\mathfrak{p}}(M_\mathfrak{p},N_\mathfrak{p}), \]
which holds for each $\mathfrak{p} \in \Spec R$, $M \in \modu R$ and $N \in \Mod R$.
\end{proof}

Recall further that a finite sequence $r_1, \dots, r_m$ of elements of a commutative ring $R$ is a \emph{regular sequence} if the coset of $r_i$ is not a zero divisor in $R/(r_1, \dots, r_{i-1})$ for each $i \in \{1, \dots, m\}$, and if $R/(r_1, \dots, r_m) \ne 0$. It is another standard result that a quotient of a regular ring by a regular sequence is Gorenstein.

\begin{prop} \label{prop:comm-Gore}
Let $R$ be a commutative noetherian ring such that $\gldim R = n < \infty$. If $r_1, \dots, r_m$ is a regular sequence, then $m \le n$ and $U := R/(r_1,\dots, r_m)$ is $(n-m)$-Gorenstein in the sense of Definition~\ref{defi:Gorenstein}.
\end{prop}

\begin{proof}
We refer to the proof of~\cite{bruns-herzog}*{Proposition 3.1.20} and Example~\ref{ex:iwanaga-gore}, but let us sketch the argument for the convenience of the reader. First one reduces the situation as above to the case when $R$ and $U$ are local. Second, we must have $\idim_R R = n$ by the proof of Proposition~\ref{prop:regular}. Finally, we use $n$ times \cite{bruns-herzog}*{Corollary 3.1.15}, which implies that $\idim_U U = \idim_R R - m$.
\end{proof}

As a direct consequence, we obtain the desired result on Gorensteinness of $\partial\cC$.

\begin{cor} \label{cor:eq-KK-Gore-boundary}
Given an arbitrary prime number $p \in \N$, the rings $\Z[t]/(1-t^p)$ and $\Z[s,t]/\big(N(s)+N(t)-p,(1-s)(1-t)\big)$ are $1$-Gorenstein.
\end{cor}

\begin{proof}
The global dimensions of $\Z[t]$ and $\Z[s,t]$ are $2$ and $3$, respectively, by~\cite{bruns-herzog}*{Theorem A.12}. Clearly $1-t^p$ is not a zero divisor, hence $1-t^p$ forms a regular sequence of length $1$. It remains to show that $r_1 := N(s)+N(t)-p, r_2 := (1-s)(1-t)$ is a regular sequence of length $2$ in $\Z[s,t]$. To this end, it is well known that $\Z[s,t]$ is a unique factorization domain, and clearly $r_1$ and $r_2$ are coprime. Now if $r_2 a + (r_1) = 0$ in $\Z[s,t]/(r_1)$ for some $a$, then $r_2 a = r_1 b$ in $\Z[s,t]$ for some $b$ and so $a \in (r_1)$.
\end{proof}

Criterion (2) of Theorem~\ref{thm:Gorenstein} is obvious since $\cC$ has finitely many objects. If $c \in \cC$, it is easy to check using relations ($\rho1$) above that both the neighbourhoods $\incoming(c)$ and $\outcoming(c)$ are given by $c$ itself and precisely one neighbour in each direction in the quiver~\eqref{eq:eq-KK-quiver}. Hence we can take $S = \id_\cC$ for Theorem~\ref{thm:Gorenstein}(3). In order to establish Theorem~\ref{thm:Gorenstein}(4) and~(5), we need some further preliminaries on finitely generated Gorenstein projective modules over commutative noetherian rings. Here we use a variant of a result due to Eisenbud~\cite{eisenbud:matrix-fact}. Recall that a \emph{matrix factorization of $w \in R$}, where $R$ is a commutative ring, is a pair $(A,B)$ of $n\times n$ matrices over $R$ for some $n \ge 1$ such that $AB = w\!\cdot\!I_n = BA$.

\begin{prop} \label{prop:MF}
Let $R$ be a commutative noetherian ring, $w \in R$ be a non-zero divisor, and $(A,B)$ be a matrix factorization of $w$. If we denote $U := R/(w)$, then
\begin{equation*} \label{eq:MF}
C\colon \quad
\xymatrix@1{
\cdots \ar[r] & U^n \ar[r]^-{B\cdot-} & U^n \ar[r]^-{A\cdot-} & U^n \ar[r]^-{B\cdot-} & U^n \ar[r]^-{A\cdot-} & U^n \ar[r] & \cdots
}
\end{equation*}
is a complete projective resolution of the $U$-module $M = \Img(U^n \overset{A\cdot-}\longrightarrow U^n)$. In particular, $M \in \Gproj U$.
If, moreover, $R$ is local and regular, then each $M \in \Gproj U$ arises in this way.
\end{prop}

\begin{proof}
Clearly $C$ is a complex which is acyclic by~\cite{eisenbud:matrix-fact}*{Proposition 5.1}. If we apply $\Hom_U(-,U)$ to $C$, we get a complex of the same type corresponding to matrix factorization $(A^t,B^t)$. Since all components of $C$ are finitely generated free modules, it follows that $\hom_U(C,P)$ is an acyclic complex for each $P \in \Proj R$. In particular, $C$ is a complete projective resolution of $M$. The last statement is a special case of \cite{eisenbud:matrix-fact}*{Theorem 6.1}.
\end{proof}

\begin{remark} \label{rem:MF-duality}
As noted above, if $M \in \Gproj U$ is obtained from a matrix factorization $(A,B)$, then $M^\vee = \Hom_U(M,U)$ is associated with the matrix factorization $(A^t,B^t)$. Since the roles of $M$ and $M^\vee$ are symmetric, there is an isomorphism $M \cong (M^\vee)^\vee$ given by the canonical pairing $\beta\colon M^\vee \otimes_U M \to U$, $f \otimes m \mapsto f(m)$.

In terms of presentations, $M \cong U^n / \Img (B\cdot-)$ and $M^\vee \cong U^n / \Img (B^t\cdot-)$, and the canonical basis of $U^n$ induces generating sets $m_1, \dots, m_n \in M$ and $f_1, \dots, f_n \in M^\vee$. If we denote by $a_{ij} \in U$ the images of the entries $A$ under the quotient map $\pi\colon R \to U$, the pairing above takes form
\begin{equation} \label{eq:MF-pairing-gen}
\beta\colon f_i \otimes m_j \longmapsto a_{ij} \in U. 
\end{equation}

In what comes, we will encounter $1\times 1$ matrix factorizations $(a,b)$, i.e.\ ordinary factorizations $w = ab$ in $R$. In such a case $M = mR$ and $M^\vee = fR$ are cyclic modules, $M \cong M^\vee$, and formula~\eqref{eq:MF-pairing-gen} for the pairing specializes to a form which is going to be very useful in the proof of Theorem~\ref{thm:eq-KK-Gore}:
\begin{equation} \label{eq:MF-pairing}
\beta\colon f \otimes m \longmapsto a \in U.
\end{equation}
\end{remark}

Now we can finish the proof that our subcategory $\cC \subseteq \KK^{C(p)}$ is $1$-Gorenstein.

\begin{proof}[Proof of Theorem~\ref{thm:eq-KK-Gore}]
Conditions (1)--(3) of Theorem~\ref{thm:Gorenstein} have been checked in the above discussion, along with fixing $S = \id_\cC$.

Let us verify condition (5). First note that the group of automorphisms of quiver~\eqref{eq:eq-KK-quiver} is the Klein group, the non-trivial symmetries being $\Sigma$ and the reflections along the vertical and horizontal axes. Since relations ($\rho1$)--($\rho5$) are also invariant under these symmetries, we get $\Aut(\cC) \cong \Z_2 \times \Z_2$. Taking into account that we always have ${_c\cC_c} \in \Gproj{_c\cC_c}$, the symmetries reduce the list of non-trivial Hom-bimodules of $\cC$ to check to three items:
\begin{itemize}
 \item[(i)]   ${_{A_1}\cC_{A_0}}$ must belong to $\Gproj{_{A_0}\cC_{A_0}}$,
 \item[(ii)]  ${_{A_2}\cC_{A_0}}$ must belong to $\Gproj{_{A_0}\cC_{A_0}}$,
 \item[(iii)] ${_{A_0}\cC_{A_2}}$ must fall in $\Gproj{_{A_2}\cC_{A_2}}$.
\end{itemize}
In case (i), the underlying abelian group of ${_{A_1}\cC_{A_0}}$ is by~\cite{koehler}*{Lemma 11.11} free of rank one, generated by ${_1\alpha_0}$. Now $\id_{A_0} - {_0t_0} \in {_{A_0}\cC_{A_0}}$ annihilates ${_{A_1}\cC_{A_0}}$ since
\[ 0 = {_1\alpha_0} \circ {_0\alpha_2} \circ {_2\alpha_0} = {_1\alpha_0} \circ (\id_{A_0} - {_0t_0}) \]
be relations ($\rho1$) and ($\rho3$). Hence up to isomorphism, we are asking whether $M = U/(1-t)$ is finitely generated Gorenstein projective over $U = \Z[t]/(1-t^p)$. This is clear, however, since $M$ arises as in Proposition~\ref{prop:MF} from the matrix factorization $(N(t), 1-t)$ of size $1\times1$ of the non-zero divisor $(1-t^p) \in R = \Z[t]$.

In case (ii), ${_{A_2}\cC_{A_0}}$ has free underlying group of rank $p-1$ by~\cite{koehler}*{Proposition 11.23} and it is annihilated by $N({_0t_0}) \in {_{A_0}\cC_{A_0}}$ because of relations ($\rho1$) and ($\rho2$). Hence the module in question is $M = U/\big( N(t) \big)$ over $U = \Z[t]/(1-t^p)$, and it comes from the matrix factorization $(1-t, N(t))$ of $(1-t^p) \in R = \Z[t]$.

Finally in case (iii) the module again has free underlying group of rank $p-1$ by~\cite{koehler}*{Proposition 11.23} and it is annihilated by $\id_{A_2} - {_2s_2} \in {_{A_2}\cC_{A_2}}$ because of relations ($\rho1$) and ($\rho4$). Thus, up to isomorphism, we have the module $M = U/(1-s)$ over $U = \Z[s,t]/\big(N(s)+N(t)-p,(1-s)(1-t)\big)$. It is Gorenstein projective since it comes from the matrix factorization $(1-t, 1-s)$ of the non-zero divisor $(1-s)(1-t) \in R = \Z[s,t]/\big(N(s)+N(t)-p\big)$.

Last but not least, we focus on condition (4) of Theorem~\ref{thm:Gorenstein}. Here we use the advantage of Remark~\ref{rem:trace-boundary} and specify the maps $\mu_d$ first. We simply put $\mu_d = \id_{_d\cC_d}$ for each $d \in \cC$, the resulting homomorphisms
\begin{align*}
\psi_{c,d}\colon {}_{d}\cC_{c} &\longrightarrow \Hom_{{}_{d}\cC_{d}}({}_{c}\cC_{d}, {}_{d}\cC_{d}) \\
f &\longmapsto f \circ -.
\end{align*}
satisfying all the required naturality properties. It remains to show that $\psi_{c,d}$ are isomorphisms for each $c,d \in \cC$. This follows from the easily verifiable fact that the associated pairing
\begin{align*}
\beta\colon {}_{d}\cC_{c} \otimes_{{}_{c}\cC_{c}} {}_{c}\cC_{d} &\longrightarrow {}_{d}\cC_{d} \\
f \otimes g &\longmapsto \psi_{c,d}(f)(g) = \mu_d(S(f)g) = f\circ g
\end{align*}
coincides with the pairing~\eqref{eq:MF-pairing} from Remark~\ref{rem:MF-duality}.
\end{proof}

\subsection{Brown-Adams representability in KK-theory}

We end our article by applying our countable variant of Brown-Adams representability, Theorem~\ref{thm:countable_brown-adams}, to the KK-theoretic setting.
Although not surprising, this result appears to be new. The representability of countable cohomological functors (corresponding below to the special case of the identity maps $\alpha= \id_H$) was first noticed by Meyer and Nest~\cite{meyer_nest_bc}.

Let $\cT$ be a Kasparov category of separable C*-algebras. Here we may equally consider the ordinary category $\KK$, or any of the (triangulated) variants in use: the equivariant Kasparov category $\KK^G$ for $G$ a locally compact group~\cite{kasparov:equivariant} or groupoid~\cite{legall:groupoids}; the Kasparov category $\KK(T)$ of C*-algebras over a topological space~$T$~\cite{meyer-nest:filtrated1}; etc.
Although these categories admit arbitrary countable coproducts, they do not seem to enjoy any nice generation property. It is therefore common practice to consider a suitable ``Bootstrap subcategory'', i.e., the localizing subcategory $\cB:=\Loc(\cG)=\Loc_{\aleph_1}(\cG)\subset \cT$ generated by some small set $\cG$ of compact objects, and typically $\mathcal G$ is a countable category. This covers for instance the following examples:
\begin{itemize}
\item The classical Bootstrap category $\mathcal B\subseteq \KK$ of Rosenberg-Schochet.
\item The Meyer-Nest bootstrap class $\mathcal B\subseteq \KK(T)$ of ~\cite{meyer-nest:filtrated1}.
\item The equivariant Bootstrap category $\mathcal B\subset \KK^G$ used in \cite{koehler} and~\cite{ivo:mackey}.
\item The larger equivariant Bootstrap category $\mathcal B\subset \KK^G$ considered in~\cite{dem:lefschetz}.
\end{itemize}
Since $\mathcal B$ lacks arbitrary coproducts, the usual version of Brown-Adams representability, Theorem~\ref{thm:brown-adams}, cannot be applied here -- indeed, the conclusion would be wrong. 
The correct result looks as follows:

\begin{thm} \label{thm:KK_brown_adams}
Let $\mathcal B =\Loc_{\aleph_1}(\mathcal G)$ be a bootstrap triangulated category of separable C*-algebras as above, where the full subcategory $\mathcal G \subseteq \mathcal B^c$ of generators has at most countably many objects and maps. 
Then every natural transformation $\alpha\colon H\to H'$ between two cohomological functors $H,H'\colon (\mathcal B^c)^\op\to \Mod \Z$ is represented by a  map $X\to X'$ of $\mathcal B$, provided $H(\Sigma^iC)$ and $H'(\Sigma^iC)$ are countable for all $C\in \mathcal G$ and $i\in \Z/2\Z$.
\end{thm}

\begin{proof}
This is an immediate consequence of Theorem~\ref{thm:countable_brown-adams}, once we have shown that all cohomological functors $H\in \Mod \mathcal B^c$ as in the theorem are countably generated. 

Assume that $H(C)$ is a countable abelian group for every suspension of a generator $C\in \mathcal G$.
As $H$ is cohomological, the class of objects $X$ for which $H(X)$ is countable is closed under the formation of cones and countable coproducts, so it contains all of~$\mathcal B$ and \emph{a fortiori} also~$\mathcal B^c$.
Similarly, since $\mathcal G$ is assumed to be countable, an induction on the length of the objects $C,D\in \mathcal B^c$ simultaneously shows that $\mathcal B(C,D)$ is countable for all $C,D\in \mathcal B^c$ and that there are only countably many isomorphism classes in $\mathcal B^c$ (cf.\ Example~\ref{ex:countable_rings}). 

Now let $\{C_i\}_{i\in I}$ be a countable set of representatives for the isomorphism classes of objects in~$\mathcal B^c$. Every element $x\in H(C_i)$ corresponds by Yoneda to a map $h_{\mathcal B^c}C_i \to H$, and the resulting map $\coprod_{i\in I, x\in H(C_i)} h_{\mathcal B^c}C_i \to H$ is an epimorphism onto $H$ from a countable coproduct of representable $\mathcal B^c$-modules.
\end{proof}

%%%%%%%%  Bibliography  %%%%%%%%%%%%%%%%%%%%%%%%%%%%%%%%%%%%%%%%%%%%%%%%%%%%%%%%%%%%%%%%%%%%%%%%%%%%%%%%%%%%%%%%%%%%%%%%%%%%%%%%%%%

\end{document}